\documentclass[a4paper,12pt]{article}
\usepackage{latexsym,amsmath,amsthm,amssymb}
\usepackage{a4wide}
\usepackage{hyperref}
\usepackage{marginnote}
\usepackage{color}
\usepackage{cite}
\hypersetup{
pdftitle={Stability for $L^2$ weighted CKN inequality}   
pdfauthor={Van Hoang},
colorlinks = true,
linkcolor = magenta,
citecolor = blue,
}
\usepackage[sort&compress]{natbib}

\theoremstyle{plain}
\newtheorem{theorem}{Theorem}[section]

\newtheorem{lemma}[theorem]{Lemma}
\newtheorem{corollary}[theorem]{Corollary}

\theoremstyle{definition}

\theoremstyle{remark}

\allowdisplaybreaks

\renewcommand{\thefootnote}{\arabic{footnote}}

\def\R{\mathbb R}

\def\de{\delta}

\def\lam{\lambda}
\def\ep{\epsilon}
\def\na{\nabla}
\def\pa{\partial}
\def\lt{\left}
\def\rt{\right}

\def\i0i{\int_0^\infty}

\def\irn{\int_{\R^n}}

\allowdisplaybreaks

\numberwithin{equation}{section}
\usepackage{color}

\title{The weighted $L^2$-Caffarelli-Kohn-Nirenberg inequalities for the curl-free vector fields and second order derivatives: The sharp constants and stability estimates.}
\author{Anh Tuan Duong\footnote{Faculty of  Mathematics and Informatics, Hanoi University of Science and Technology, 1 Dai Co Viet, Bach Mai, Ha noi, Viet Nam.}\, and Van Hoang Nguyen\footnote{Department of Mathematics, FPT University, Ha Noi, Viet Nam.} 
}

\begin{document}
\maketitle


\renewcommand{\thefootnote}{}

\footnote{Email: \href{mailto: Anh Tuan Duong<tuan.duonganh@hust.edu.vn>}{tuan.duonganh@hust.edu.vn}; \href{mailto:Van Hoang Nguyen <hoangnv47@fe.edu.vn>}{hoangnv47@fe.edu.vn}.}

\footnote{2010 \emph{Mathematics Subject Classification\text}: 26D10, 46E35, 26D15}

\footnote{\emph{Key words and phrases\text}: Caffarelli-Kohn-Nirenberg inequality, sharp constants, extremal functions, stability version, spherical harmonic decomposition.}

\renewcommand{\thefootnote}{\arabic{footnote}}
\setcounter{footnote}{0}

\begin{abstract}
In this paper, we study the weighted $L^2$-Caffarelli-Kohn-Nirenberg inequalities for curl-free vector fields and second order derivatives. Firstly, we prove a family of the sharp weighted second order $L^2$-Caffarelli-Kohn-Nirenberg inequalities that complements the results in [{\it C. Cazacu, J. Flynn and N. Lam,
	Calc. Var. Partial Differential Equations 62 (2023), no. 4, Paper No. 118, 26 pp.}] and [{\it A. T. Duong and V. H. Nguyen, On the sharp second order Caffarelli-Kohn-Nirenberg
inequality. Ann. Fenn. Math., 50(1):275--286, 2025}]. Secondly, we establish a stability version of the sharp weighted $L^2$-Caffarelli-Kohn-Nirenberg inequalities for curl-free vector fields proved by Cazacu, Flynn and Lam. Finally, we prove a stability estimate for the sharp weighted second order $L^2$-Caffarelli-Kohn-Nirenberg inequalities established in this paper. Our approach is based on the spherical harmonic decomposition method, the one dimensional integral inequalities and their improvements.
\end{abstract}

\section{Introduction}
Let us start by recalling the first order interpolation inequalities introduced by Caffarelli, Kohn and Nirenberg in \cite{1stCKN} known as the Caffarelli-Kohn-Nirenberg (CKN for short) inequalities.

\noindent{\bf Theorem A. (CKN inequalities) }{\it  Let $n \geq 1$ and let $p, q, r, \alpha, \beta, \delta$ and $\sigma$ be real numbers satisfying
\[
p,q \geq 1,\qquad r >0,\qquad \delta \in [0,1],
\]
and
\[
\frac1p + \frac \alpha n > 0, \quad \frac 1q + \frac \beta n > 0,\quad  \frac 1r + \frac\gamma n > 0,
\]
where $ \gamma = \delta \sigma + (1-\delta)\beta$. Then there exists  $C > 0$ such that the inequality
\begin{equation}\label{eq:1stCKN}
\Big(\int_{\R^n} |u|^r |x|^{r\gamma} dx\Big)^{\frac1r} \leq C \Big(\int_{\R^n} |\nabla u|^p |x|^{\alpha p} dx\Big)^{\frac\delta p}\Big(\int_{\R^n} |u|^q |x|^{\beta q} dx\Big)^{\frac{1-\delta}q}
\end{equation}
holds for any function $u\in C_0^\infty(\R^n)$ if and only if the following relations occur
\[
\frac1r + \frac\gamma n = \delta\Big(\frac1p + \frac{\alpha -1}n\Big) +(1-\delta)\Big(\frac1q + \frac\beta n\Big),
\]
\[
\alpha -\sigma \geq 0 \qquad\text{ if }\qquad a >0,
\]
and 
\[
\alpha -\sigma \leq 1 \qquad\text{ if } a > 0\quad\text{ and } \quad \frac 1p + \frac{\alpha -1}n = \frac 1r +\frac\gamma n.
\]
}Furthermore, on any compact set in the space of parameters, the constant $C$ above is bounded. Interestingly, the CKN inequalities \eqref{eq:1stCKN} contain many well-known and important inequalities in analysis such as Sobolev inequalities, Gagliardo-Nirenberg inequalities, Hardy inequalities, Hardy-Sobolev inequalities, Nash's inequalities, Heisenberg-Pauli-Weyl uncertainty principle, etc.  as special cases. The CKN inequalities play the important roles and have many applications in analysis and theory of partial differential equations (see, e.g, \cite{CKNappl}). They have been extensively studied by many mathematicians, especially on the questions about the sharp constant and the classification of extremal functions for the CKN inequalities. In this research direction, we can list some celebrated results of Aubin and Talenti \cite{Aubin,Talenti}, of Del Pino and Dolbeault \cite{DD1,DD2}, and of Carlen and Loss \cite{CarlenLoss}  for the sharp Sobolev inequalities, the sharp Gagliardo-Nirenberg inequalities,  and the sharp Nash's inequalities, respectively. The interest reader may consult the papers \cite{NguyenPLMS,NguyenHC1,NguyenHC2,BrezisMironescu,BrezisJeanYung,Xia,Ruzhansky1,Ruzhansky2,Ruzhansky3,CatrinaWang,ChenLuZhang,Dong,DongLamLu,Flynn,FlynnLamLu,LamLu} for other results concerning  the CKN type inequalities. In \cite{Costa,CatrinaCosta}, the following important subfamily of the sharp $L^2-$CKN inequalities was proved
\begin{equation}\label{eq:L2CKN}
\int_{\R^n} \frac{|u|^2}{|x|^{2b}} dx \int_{\R^n} \frac{|\nabla u|^2}{|x|^{2a}} dx \geq C(n,a,b)^2\lt(\int_{\R^n} \frac{|u|^2}{|x|^{a+b+1}} dx\rt)^2,\quad u \in C^\infty_0(\R^n\setminus\{0\}).
\end{equation}
The sharp constant $C(n,a,b)$ and the extremal functions in \eqref{eq:L2CKN} are explicitly computed and depend on the parameter regions of $a, b$ as follows
\begin{equation*}
\begin{cases}
\mathcal A_1 :=\{(a,b)\, |\, a+1 > b, a\leq (n-2)/2\}\\
\mathcal A_2 := \{(a,b)\, |\, a+1 < b, a\geq (n-2)/2\}\\
\mathcal A := \mathcal A_1 \cup \mathcal A_2\\
\mathcal B_1 := \{(a,b)\, |\, a+1 <b, a\leq (n-2)/2\}\\
\mathcal B_2 := \{(a,b)\, |\, a+1 > b, a\geq (n-2)/2\}\\
\mathcal B = \mathcal B_1 \cup \mathcal B_2.
\end{cases}
\end{equation*}
More precisely, it was shown in \cite{CatrinaCosta} that

$\bullet$ If $(a,b) \in \mathcal A$, then the sharp constant $C(n,a,b) = |n -(a+b +1)|/2$ and the extremal functions are given by 
\[
u(x) = c \exp\Big(-\lambda \frac{|x|^{1+a -b}}{1+a-b}\Big),
\]
where $c, \lambda $ are constants with $\lambda (1+a -b) > 0$.

$\bullet$ If $(a,b) \in \mathcal B$, then the sharp constant  $C(n,a,b) = |n -(3a -b +3)|/2$ and the extremal functions are given by 
\[
u(x) = c |x|^{2(a+1) -n}\exp\Big(-\lambda \frac{|x|^{1+a -b}}{1+a-b}\Big)
\]
where $c, \lambda $ are constants with $\lambda (1+a -b) > 0$.

 In the case $1+a = b$, the CKN inequality \eqref{eq:L2CKN} becomes the weighted Hardy inequality with the sharp constant given by $C(n,a,a+1) = |n -2a -2|/2$. In this case, the sharp constant $C(n,a,a+1)$ is not attained by non-zero functions. The inequality \eqref{eq:L2CKN} was first proved by Costa in \cite{Costa} for a particular range of parameters by using the \emph{\text expanding-the-square} method, and then was established by Catrina and Costa in \cite{CatrinaCosta} for full range of parameters by using the spherical harmonic decomposition method and the Kelvin transformation. A simple and direct proof of \eqref{eq:L2CKN} was given in \cite{CaFlyLa}. Moreover, in \cite{CaFlyLa}, the authors obtained a stronger version of \eqref{eq:L2CKN} in which the full gradient $\nabla u$ is replaced by the radial derivative $\nabla u\cdot \frac x{|x|}$.

It is remarkable that in the case $a =0$ and $b =-1$, the CKN inequality \eqref{eq:L2CKN} is exactly the famous Heisenberg-Pauli-Weyl uncertainty principle (see \cite{Heisenberg,Weyl}), which asserts that 
\begin{equation}\label{eq:HPWup}
\int_{\R^n} |\nabla u|^2 dx \int_{\R^n} |u|^2 |x|^2 dx \geq \frac{n^2}4 \lt(\int_{\R^n} |u|^2 dx\rt)^2,\quad u \in C_0^\infty(\R^n).
\end{equation}
The constant $n^2/4$ in \eqref{eq:HPWup} is sharp and is attained by Gaussian functions $u(x) = c e^{-\lambda |x|^2}$ with $c \in \R$ and $\gamma > 0$. Motivated by questions and applications in hydrodynamics and harmonic analysis, Maz'ya in \cite[Section $3.9$]{Mazya} raised the question on the sharp constant in \eqref{eq:HPWup} when the scalar function $u$ is replaced by a divergence-free vector field $U$, here a vector field $U = (U_1,U_2,\ldots,U_n) : \R^n \to \R^n$ is called divergence-free if 
\[
\text{\rm div} U = \partial_1 U_1 + \partial_2 U_2 + \cdots+ \partial_n U_n = 0.
\]
This question was solved by Cazacu, Flynn and Lam (see \cite{CFL22}) in two dimensional case, i.e, $n =2$. More precisely, in \cite{CFL22} the authors established the following inequality
\begin{equation}\label{eq:2ndHUP}
\int_{\R^n} |\Delta u|^2 dx \int_{\R^n} |\nabla u|^2 |x|^2 dx \geq \frac{(n+2)^2}4 \lt(\int_{\R^n} |\nabla u|^2 dx\rt)^2,\quad u\in C_0^\infty(\R^n).
\end{equation}
The constant $(n+2)^2/4$ is sharp and is achieved by Gaussian functions $u(x) = c e^{-\lambda |x|^2}$ where $c \in \R$ and $\lambda >0$. A vector field $U$ is curl-free if $U =\nabla u$ for some scalar function $u: \R^n \to \R$. Therefore, the inequality \eqref{eq:2ndHUP} is equivalent to following inequality
\begin{equation}\label{eq:curlfreeHUP}
\int_{\R^n} |\nabla U|^2 dx \int_{\R^n} |U|^2 |x|^2 dx \geq \frac{(n+2)^2}4 \lt(\int_{\R^n} |U|^2 dx\rt)^2
\end{equation}
for any curl-free vector field $U\in (C_0^\infty(\R^n))^n$. The constant $(n+2)^2/4$  is again the best possible and it is attained by the curl-free vector fields of the form $U(x) = c x e^{-\lambda |x|^2}$ with $c\in \R$ and $\lambda >0$. The inequality \eqref{eq:curlfreeHUP} provides an uncertainty principle of Heisenberg-Pauli-Weyl type for curl-free vector fields. By an isometrically isomorphic between divergence-free vector fields and curl-free vector fields in $\R^2$ (in fact, if $U$ is a divergence-free vector field in $\R^2$ then $U(x) = (-\partial_{x_2} u, \partial_{x_1} u)$ for some scalar function $u$), the inequality \eqref{eq:curlfreeHUP} is equivalent to the following inequality in $\R^2$
\begin{equation}\label{eq:2dimcase}
\int_{\R^2} |\nabla U|^2 dx \int_{\R^2} |U|^2 |x|^2 dx \geq 4 \lt(\int_{\R^2} |U|^2 dx\rt)^2
\end{equation}
for any divergence-free vector fields $U\in (C^\infty_0(\R^2))^2$. The constant $4$ is sharp and is attained by the vector fields of the form $U(x) =c(-x_2,x_1) e^{-\lambda |x|^2}$, with $c\in \R$ and $\lambda >0$. The inequality \eqref{eq:2dimcase} answers affirmatively the question of Maz'ya in the case $n=2$. Recently, this question was completely solved by Hamamoto \cite{HamaMaz} (see also \cite{Hamasimple} for a simpler proof). In fact, by using the poloidal-toroidal decomposition for divergence-free vector fields and an one-dimensional variational problem, he proved that
\begin{equation}\label{eq:divergencefreeHUP}
\int_{\R^n} |\nabla U|^2 dx \int_{\R^n} |U|^2 |x|^2 dx \geq \frac14(\sqrt{n^2-4(n-3)} +2)^2 \lt(\int_{\R^n} |U|^2 dx\rt)^2
\end{equation}
for any divergence-free vector field $U\in (C_0^\infty(\R^n))^n$. The constant $(\sqrt{n^2-4(n-3)} +2)^2/4$ is sharp in \eqref{eq:divergencefreeHUP} and is achieved by some divergence-free vector fields.

In \cite{CFL23}, Cazacu, Flynn and Lam extended the inequality \eqref{eq:curlfreeHUP} to the weighted case and proved a family of sharp CKN inequalities for curl-free vector fields. In order to state their results, let us define $X_{a,b}(\R^n)$, with $a,b\in \R$, to be the set of vector fields $U\in C^\infty(\R^n\setminus \{0\})^n$ such that
\[
\int_{\R^n} \frac{|\nabla U|^2}{|x|^{2a}} dx < \infty,\quad \int_{\R^n} \frac{|U|^2}{|x|^{2b}}dx < \infty,\quad \int_{\R^n} \frac{|U|^2}{|x|^{a+b +1}} dx < \infty,
\]
\[
\lim_{|x|\to 0,\infty} |x|^{-1 + \frac n2 - a} U(x) = 0
\]
and
\[
\lim_{|x|\to 0, \infty} |x|^{-b + \frac n2} U(x) = 0.
\]
Let $n\geq 1$ and $a, b$ be real numbers such that $(n-2a)^2 > 4(n+1)$, it was proved by Cazacu, Flynn and Lam in \cite{CFL23} that for any curl-free vector field $U \in X_{a,b}(\R^n)$,
\begin{equation}\label{eq:CFLHUP}
\irn \frac{|\nabla U|^2}{|x|^{2a}} dx \irn \frac{|U|^2}{|x|^{2b}} dx \geq C_1(n,a,b)^2 \Big(\irn \frac{|U|^2}{|x|^{a+b +1}} dx\Big)^2.
\end{equation}
The sharp constant $C_1(n,a,b)$ and extremal vector fields $U$ are fully described and depend on the parameters $a$ and $b$. More precisely, they proved for $n\geq 1$ and $a,b\in \R$,  $(n-2a)^2 \geq 4(n+1)$ that  if $a -b + 1 > 0$ then 
\[
C_1(n,a,b) = \frac{a-b +1 + \sqrt{(n-2a -2)^2 + 4(n-1)}}2,
\]
and it is achieved by the curl-free vector fields
\[
U(x) = c |x|^{-\frac{n -2a - \sqrt{(n-2a -2)^2 + 4(n-1)}}2} e^{-\frac{\lambda}{a-b+1} |x|^{a-b +1}} x,\quad c\in \R, \, \lambda >0.
\]
In the case $a-b +1 < 0$, the sharp constant $C_1(n,a,b)$ is given by
\[
C_1(n,a,b) = -\frac{a-b +1 - \sqrt{(n-2a -2)^2 + 4(n-1)}}2.
\]
and is attained by the curl-free vectors fields
\[
U(x) = c |x|^{-\frac{n -2a + \sqrt{(n-2a -2)^2 + 4(n-1)}}2} e^{\frac{\lambda}{a-b+1} |x|^{a-b +1}} x,\quad c \in \R,\, \lambda >0.
\]
The inequality \eqref{eq:CFLHUP} becomes the sharp weighted Hardy inequality for curl-free vector fields of Hamamoto and Takahashi \cite{HamaTaka1} by letting $b\to a+1$ (see also \cite{HamaTaka2} for another proof). It also contains the weighted Heisenberg-Pauli-Weyl uncertainty principle for curl-free vector fields if $b = -a -1$, $a > -1$, and the weighted Hydrogen uncertainty principle if $b = -a$, $a > -\frac 12$. In particular, when $a =0$ and $n\geq 5$,  the Heisenberg-Pauli-Weyl uncertainty principle and the Hydrogen uncertainty principle for curl-free vector fields were proved in \cite{CFL22}.

Let $u$ be the scalar potential of $U$, i.e, $U = \nabla u$. Then,  the inequality \eqref{eq:CFLHUP} is equivalent to the following inequality
\begin{equation*}
\irn \frac{|\nabla^2u|^2}{|x|^{2a}} dx \irn \frac{|\nabla u|^2}{|x|^{2b}} dx \geq C_1(n,a,b)^2 \Big(\irn \frac{|\nabla u|^2}{|x|^{a+b +1}} dx\Big)^2
\end{equation*}
for any function $u$ with $\nabla u \in X_{a,b}(\R^n)$. Furthermore, the equality holds if
\[
\nabla u(x) = c |x|^{-\frac{n -2a - \sqrt{(n-2a -2)^2 + 4(n-1)}}2} e^{-\frac{\lambda}{a-b+1} |x|^{a-b +1}} x,\quad c\in \R,\, \lambda >0
\]
in the case $a -b + 1 >0$, and if
\[
\nabla u(x) = c |x|^{-\frac{n -2a +\sqrt{(n-2a -2)^2 + 4(n-1)}}2} e^{\frac{\lambda}{a-b+1} |x|^{a-b +1}} x,\quad c\in \R,\, \lambda > 0
\]
in the case $a - b + 1 < 0$.

As mentioned in \cite{CFL23}, it is not true in general that for $U = \nabla u$
\[
\int_{\R^n} \frac{|\nabla U|^2}{|x|^{2a}} dx = \int_{\R^n} \frac{|\Delta u|^2}{|x|^{2a}} dx.
\]
Hence, it is worth-while to prove a version of \eqref{eq:CFLHUP} with $\int_{\R^n} \frac{|\Delta u|^2}{|x|^{2a}} dx$ instead of $\int_{\R^n} \frac{|\nabla U|^2}{|x|^{2a}} dx$. Hence, one proposed  to prove the following weighted $L_2$ CKN inequality for second order derivatives
\begin{equation}\label{eq:2ndCKN}
\irn \frac{|\Delta u|^2}{|x|^{2a}} dx \irn \frac{|\nabla u|^2}{|x|^{2b}} dx \geq C_2(n,a,b)^2 \Big(\irn \frac{|\nabla u|^2}{|x|^{a+b +1}} dx\Big)^2.
\end{equation}
This inequality has been only proved in  \cite{DNAFM} for the special case of parameters $a,b$:  $b = -a -1$ if $1 + a >0$. The sharp constant  is given by 
\[
C_2(n,a,-a-1) = \frac{n +4a +2}2,
\]
which is attained by the functions of form $u(x) = c e^{-\lambda |x|^{2(1+a)}}$ with $c\in \R$ and $\lambda >0$. In fact, in \cite{DNAFM}, a stronger inequality was established in which $\irn \frac{|\nabla u|^2}{|x|^{2b}} dx$ was replaced by $\irn \frac{|\nabla u \cdot \frac x{|x|}|^2}{|x|^{2b}} dx$. This result was reproved in \cite{CFL23} by using another method. Moreover, a family of non-radial extremals was found in \cite{CFL23}.

Our first goal of this paper is to prove \eqref{eq:2ndCKN} for more general parameters $a$ and $b$. To do this, let us introduce the space $Y_{a,b}(\R^n)$,  which is the set of all functions in $C^\infty(\R^n\setminus\{0\})$ such that
\[
\irn \frac{|\Delta u|^2}{|x|^{2a}} dx <\infty, \quad \irn \frac{|\nabla u|^2}{|x|^{2b}} dx < \infty,\quad \irn \frac{|\nabla u|^2}{|x|^{a+b+1}} dx < \infty,
\]
\[
\lim_{|x|\to 0,\infty} |x|^{-1-a +\frac n2} \nabla u(x) = 0
\]
and
\[
\lim_{|x|\to 0,\infty} |x|^{-b +\frac n2} \nabla u(x) = 0.
\]

Our first main result of this paper reads as follows.

\begin{theorem}\label{MainII}
Let $n\geq 2$ and $a, b$ be real numbers such that $1-b+a \not=0$, $n-2a -4\not=0$, and $(n-2a -2)^2 \geq 8(1+a)(1+2a)$. Then, for any function $u\in Y_{a,b}(\R^n)$, we have
\begin{equation}\label{eq:2ndCKNnew}
\irn \frac{|\Delta u|^2}{|x|^{2a}} dx \irn \frac{|\nabla u|^2}{|x|^{2b}} dx \geq C_2(n,a,b)^2 \Big(\irn \frac{|\nabla u|^2}{|x|^{a+b +1}} dx\Big)^2
\end{equation}
where
\[
C_2(n,a,b) = \frac{|a-b +1| + |n+2a|}2.
\]
The constant $C_2(n,a,b)$ is sharp and is achieved by $u = c \varphi(\lambda x)$ with $c\in \R$ and $\lambda >0$, where the function $\varphi$ is defined by
\begin{equation}\label{eq:extremavphi}
\varphi(x) =
\begin{cases}
\int_{|x|}^\infty r^{ -\frac{ n - 2a -2 - |n+2a|}2} e^{-\frac{r^{a-b+1}}{a-b+1}} dr&\mbox{if $n > 2a +4$ and $1-b+a >0$,}\\
\int_0^{|x|} r^{ -\frac{ n - 2a -2 - |n+2a|}2} e^{-\frac{r^{a-b+1}}{a-b+1}} dr&\mbox{if $n < 2a +4$ and $1-b+a >0$,}\\
\int_0^{|x|} r^{ -\frac{ n - 2a -2 - |n+2a|}2} e^{\frac{r^{a-b+1}}{a-b+1}} dr&\mbox{if $n > 2a +4$ and $1-b+a <0$,}\\
\int_{|x|}^\infty r^{ -\frac{ n - 2a -2 - |n+2a|}2} e^{\frac{r^{a-b+1}}{a-b+1}} dr&\mbox{if $n< 2a +4$ and $1-b+a <0$}.
\end{cases}
\end{equation}
\end{theorem}
A direct consequence of Theorem \ref{MainII} is the following sharp weighted Hydrogen uncertainty principle for the second order derivative in the case $b=-a$. 
\begin{corollary}
	Let $n\geq 2$, $a > -\frac 12$ and $(n-2a -2)^2 \geq 8(1+a)(1+2a)$. Then for any function $u\in Y_{a,-a}(\R^n)$, we have
\[
\irn \frac{|\Delta u|^2}{|x|^{2a}} dx \irn |\nabla u|^2 |x|^{2a} dx \geq  { \frac{(n+4a+1)^2}4} \Big(\irn \frac{|\nabla u|^2}{|x|} dx\Big)^2.
\]
\end{corollary}
In particular, when  $a=0$, we recover the second order Hydrogen uncertainty principle in \cite{CFL22}.

In the paper \cite{BL1985}, Br\'ezis and Lieb posed the question on the stability estimate for the sharp Sobolev inequality by asking that how close the function is to the set of extremals when the equality nearly occurs. This question can be seen as a quantitative version of the celebrated concentration-compactness principle of Lions for the Sobolev inequality \cite{LionsI}, and was affirmatively answered by Bianchi and Egnell in \cite{BiEg91}. In fact, by using a compact argument based on the concentration-compactness principle and a local stability estimate of Sobolev inequality via a spectral problem, Bianchi and Egnell proved that for $n\geq 3$, there exists a positive constant $C_{BE} >0$ such that
\begin{equation}\label{eq:BEstab}
\int_{\R^n} |\nabla u|^2 dx - S_n^2 \lt(\int_{\R^n} |u|^{\frac{2n}{n-2}} dx\rt)^{\frac{n-2}n} \geq C_{BE} \inf_{v \in \mathcal M} \int_{\R^n} |\nabla u-\nabla v|^2 dx,\quad u\in C_0^\infty(\R^n),
\end{equation}
where $S_n$ and $\mathcal M$ are the sharp constant and the set of extremals in the Sobolev inequality, respectively (see \cite{Aubin,Talenti}). Since then,  the stability estimates for functional inequalities in analysis and geometry have attracted a lot of attention by many mathematicians and many stability results for functional inequalities were established, for example  \cite{NguyenAiM,NguyenGNS,DNCVPDE,BDS26,BDNS25,BDS24,BDS24bis,DT16,DT13,DJ14,BFR24,FZ22,FZ22bis,FN19,Neu,FJ17,FJ15,FMP13,FiCar,FMP10,CLLS,Carlen25,Carlen17,CFL14,CFFP,FMP08,FMP09,CiFe08,Frank24,FrPe24,Frank22,CFW,CFLL24,Seuffert,WeiWu24,WeiWu22,FaInLe,MVNonA,ChenLuTang23,ChenLuTang24,ChenLuTang25,ChenLuTang26,ChenLuTangWang26,InKim}.  In recent years, there has been  great attention to the explicit estimate for the constant $C_{BE}$ in \eqref{eq:BEstab}. In \cite{DEFFL25}, Dolbeault, Esteban, Figalli, Frank and Loss established a lower bound with optimal dimensional dependence for $C_{BE}$ by using the competing symmetries, a flow based on continuous Steiner symmetrization that interpolates continuously between a function and its symmetric non-increasing rearrangement, and refined estimates for Sobolev functional in the neighborhood of $\mathcal M$. In the papers \cite{ChenLuTang24,ChenLuTang25,ChenLuTang26}, Chen, Lu, and Tang obtained the explicit lower bounds for the stability constants in the Hardy-Littlewood-Sobolev and the fractional Sobolev inequalities on the Euclidean space. In \cite{ChenLuTangWang26}, Chen, Lu, Tang and Wang developed a new approach based on the CR Yamabe flow to obtain the asymptotically sharp stability estimate for the Sobolev inequality on the Heisenberg group. We refer to the paper \cite{Konig} in which Konig established a further result on the attainability of the constant $C_{BE}$ in \eqref{eq:BEstab}.


For the family of $L^2-$CKN inequalities \eqref{eq:L2CKN}, the first stability estimate was proved by McCurdy and Venkatraman in \cite{MVNonA} for the case $a=0$ and $b=-1$, i.e, they established the stability version for the Heisenberg-Pauli-Weyl uncertainty principle \eqref{eq:HPWup}. More concretely, based on the approach of Bianchi and Egnell, McCurdy and Venkatraman proved the existence of two positive constants $C_1$ and $C_2$ such that
\begin{align*}
\int_{\R^n} |\nabla u|^2 dx &\int_{\R^n} |u|^2 |x|^2 dx - \frac{n^2}4 \lt(\int_{\R^n} |u|^2 dx\rt)^2\\
& \geq C_1 \lt(\int_{\R^n} |u|^2 dx\rt) \inf_{v\in E_{HUP}} \lt(\int_{\R^n} |u -v|^2 dx\rt) + C_2 \inf_{v\in E_{HUP}}\lt(\int_{\R^n} |u-v|^2 dx\rt)^2
\end{align*}
where $E_{HUP}$ is the set of Gaussian functions, i.e, 
\[
E_{HUP}= \{c e^{-\alpha |x|^2}\, :\, \alpha >0, \, c\in \R\}.
\]
In \cite{Fathi}, Fathi gave a simple proof of the above estimate by using the Poincar\'e inequality for Gaussian type measures. Moreover, he obtained the explicit constants $C_1 = \frac14$ and $C_2 = \frac1{16}$ which are not optimal. Later, Cazacu, Flynn, Lam and Lu in \cite{CFLL24} show that the sharp values of these constants must be $C_1 = n$ and $C_2 =1$. In addition, these authors in \cite{CFLL24} also proved the stability estimates for the CKN inequalities \eqref{eq:L2CKN} with parameters $a,b$ satisfying the conditions
\[
0\leq a\leq \frac{n-2}2,\,\, b \leq \frac{na}{n-2},\,\, \text{ and } \,\, a +b + 1 =\frac{2nb}{n-2}.
\] 
The restricted constraints on $a$ and $b$ are necessary to apply the Poincar\'e inequality for the log concave measures after making the change of variables. These restrictions on $a$ and $b$ were removed recently in \cite{DLLN26}, where Do, Lam, Lu and the second author established the sharp stability version of the CKN inequality \eqref{eq:L2CKN} for full ranges of $a$ and $b$. 

Although the stability version of the $L^2-$CKN inequalities \eqref{eq:L2CKN} was completely characterized, however there are very little results on the stability estimates for the $L^2-$CKN inequalities for the curl-free vector fields \eqref{eq:CFLHUP}. In \cite{DNCVPDE}, the authors in this paper proved the following stability estimate for the second order uncertainty principle established in \cite{CFL22}, that is the inequality \eqref{eq:CFLHUP} with  $a =0$ and $b=-1$.

\noindent{\bf Theorem B. }{\it For any $n\geq 2$, it holds for any non-zero function $u\in H^2_{0,-2}(\mathbb R^n)$ that
\begin{align}\label{eq:DNstab}
&\frac{\lt(\int_{\R^n} |\Delta u|^2 dx\rt)^{\frac12}\lt(\int_{\R^n} |\nabla u|^2 |x|^2 dx\rt)^{\frac12} -\frac{n+2}2 \int_{\R^n} |\nabla u|^2 dx}{\int_{\R^n} |\nabla u|^2 dx} \notag\\
&\quad\qquad\qquad \geq \frac1{768}\inf_{v \in E_{HUP}}\lt\{\frac{\int_{\R^n}|\nabla u - \nabla v|^2 dx}{\int_{\R^n} |\nabla u|^2 dx}\,:\, \int_{\R^n} |\nabla v|^2 dx = \int_{\R^n} |\nabla u|^2 dx\rt\}.
\end{align}
}To prove \eqref{eq:DNstab}, we first established an improvement of the second order uncertainty principle for functions that are orthogonal to all radial functions. With the help of this improvement, if the left hand side of \eqref{eq:DNstab} is small, then the function $u$ is close to an even function. Then, by employing the spectral analysis for Ornstein-Uhlenbeck operator associated to Gaussian measure and Hermite polynomials, we obtained the stability result for even functions. The inequality \eqref{eq:DNstab} follows from these two ingredients. However, the constant $\frac1{768}$ is far from being sharp. In \cite{DLL26}, Do, Lam, Lu and Zhang gave another proof of the above result by using the spherical harmonic decomposition technique and Fourier transformation for radial functions. They also got a better estimate by showing that the above result holds with constant $\frac14(\sqrt{n^2 + 4n -4} -n)$. Later, this constant was proved to be sharp by Huang and Ye in \cite{HuangYe}. 

Therefore, the second goal in this paper is to establish the stability version of the $L^2-$CKN inequalities for the curl-free vector fields \eqref{eq:CFLHUP} and for the second order derivative from Theorem \ref{MainII}. To state our next main results, let us introduce some notations. For $a, b\in \R$, we define by $U_{a,b}$ the curl-free vector field
\begin{equation*}
U_{a,b}(x) =
\begin{cases}
|x|^{ -\frac n2 + a  +\frac12 \sqrt{(n-2a -2)^2 + 4(n-1)}} e^{-\frac{|x|^{a -b +1}}{a -b +1}} x &\mbox{if $1-b +a >0$}\\
|x|^{ -\frac n2 + a  -\frac12 \sqrt{(n-2a -2)^2 + 4(n-1)}} e^{\frac{|x|^{a -b +1}}{a -b +1}}x &\mbox{if $1-b +a < 0$,}
\end{cases}
\end{equation*}
and 
\[
\mathcal M_{a,b} =\{ c U_{a,b}(\lambda x)\, :\, c \in \R,\, \lambda >0\}.
\]
Note that $U_{a,b}$ is extremal curl-free vector field for the $L^2-$CKN inequality \eqref{eq:CFLHUP} and $\mathcal M_{a,b}$ is the set of extremal curl-free vector fields. For a curl-free vector field $U\in X_{a,b}(\R^n)$, we define the CKN deficit $\delta(U)$ by
\[
\delta(U) = \frac{\Big(\irn \frac{|\nabla U|^2}{|x|^{2a}} dx\Big)^{\frac12}\Big( \irn \frac{|U|^2}{|x|^{2b}} dx\Big)^{\frac12}}{\irn \frac{|U|^2}{|x|^{a+b+1}} dx} -C_1(n,a,b) 
\]
if $U\not\equiv 0$ and $\de(U) =0$ if $U\equiv 0$. Our second main result is a stability version of the $L^2-$CKN inequality for curl-free vector fields \eqref{eq:CFLHUP} and is stated as follows
\begin{theorem}\label{MainI}
Let $n\geq 2$ and $a,b \in \R$  such that $1-b+ a\not =0$, $n-2a -4\not=0$, and $(n-2a)^2  > 4(n+1)$. Then, for any curl-free vector field $U\in X_{a,b}(\R^n)$, it holds
\begin{equation*}
\delta(U) \geq C \inf\left\{\frac{\irn \frac{|U - V|^2}{|x|^{a+b+1}} dx}{\irn \frac{|U|^2}{|x|^{a+b+1}} dx}\, :\, V \in \mathcal M_{a,b},\quad \irn \frac{|U|^2}{|x|^{a+b+1}} dx = \irn \frac{|V|^2}{|x|^{a+b+1}} dx\right\}
\end{equation*}
for some positive constant $C$ depending only on $n,a,b$.
\end{theorem}
Remark that the constant $C$ can be computed explicitly from the proof of Theorem \ref{MainI}, however it is not optimal.  We now give some consequences of Theorem \ref{MainI}. 
\begin{corollary}
	Let $b = -a -1$, $a > -1$ and $(n-2a)^2 > 4(n+1)$. Then,  we have the following weighted Heisenberg-Pauli-Weyl uncertainty principle for curl-free vector fields:
\[
\delta(U) \geq C \inf\left\{\frac{\irn |U - V|^2 dx}{\irn |U|^2 dx}\, :\, V \in \mathcal M_{a,-a-1},\quad \irn |U|^2 dx = \irn |V|^2 dx\right\}
\]
for any curl-free vector field $U\in X_{a,-a-1}(\R^n)$.
\end{corollary} 
\begin{corollary}
	Let $b = -a$, $a  > -\frac12$ and $(n-2a)^2 > 4(n+1)$. Then,  we have the following weighted Hydrogen uncertainty principle for curl-free vector fields:
\[
\delta(U) \geq C \inf\left\{\frac{\irn \frac{|U - V|^2}{|x|} dx}{\irn \frac{|U|^2}{|x|} dx}\, :\, V \in \mathcal M_{a,-a},\quad \irn \frac{|U|^2}{|x|} dx = \irn \frac{|V|^2}{|x|} dx\right\}
\]
for any curl-free vector field $U\in X_{a,-a}(\R^n)$.

\end{corollary}
Put 
\begin{equation*}
\psi(x) = 
\begin{cases}
\int_{|x|}^\infty r^{ -\frac n2 + a +1 +\frac12 \sqrt{(n-2a -2)^2 + 4(n-1)}} e^{-\frac{r^{a -b +1}}{a -b +1}} dr &\mbox{if $1-b+a >0$,}\\
\int_0^{|x|} r^{ -\frac n2 + a +1 -\frac12 \sqrt{(n-2a -2)^2 + 4(n-1)}} e^{\frac{r^{a -b +1}}{a -b +1}} dr &\mbox{if $1-b+a <0$,}
\end{cases}
\end{equation*}
and 
\[
\mathcal M_1 = \left\{c \psi(\lambda x)\, :\, c \in \R, \lambda >0 \right\}.
\]
Let $u$ be the scalar potential of the curl-free vector field $U$. From Theorem \ref{MainI}, we have the following 
\begin{corollary}
Let $n\geq 2$ and $a, b$ be  real numbers satisfying $1-b+a \not=0$, $n -2a -4 \not=0$ and  $(n-2a)^2 > 4(n+1)$. Then, for any non-zero function $u$ such that $\nabla u\in X_{a,b}(\R^n)$, we have 
\begin{align*}
&\frac{\Big(\irn \frac{|\nabla^2 u|^2}{|x|^{2a}} dx\Big)^{\frac12}\Big( \irn \frac{|\nabla u|^2}{|x|^{2b}} dx\Big)^{\frac12}}{\irn \frac{|\nabla u|^2}{|x|^{a+b+1}} dx} -C_1(n,a,b) \\
&\qquad\qquad \geq C \inf\left\{\frac{\irn \frac{|\nabla u - \nabla v|^2}{|x|^{a+b+1}} dx}{\irn \frac{|\nabla u|^2}{|x|^{a+b+1}} dx}\, :\, v \in \mathcal M_{1},\quad \irn \frac{|\nabla u|^2}{|x|^{a+b+1}} dx = \irn \frac{|\nabla v|^2}{|x|^{a+b+1}} dx\right\},
\end{align*}
where $\nabla^2 u$ denotes the Hessian of $u$ and $|\nabla^2 u|^2 = \sum_{i,j=1}^n (\pa_{ij} u)^2$.
\end{corollary}

Our last main result concerns with  the stability estimate for $L^2-$CKN inequality for second order derivatives in Theorem \ref{MainII}. Let us denote
 \[
 \mathcal N_{a,b} = \{c \varphi(\lambda x)\, :\, c \in \R,\, \lambda >0\},
 \]
where $\varphi$ is defined by \eqref{eq:extremavphi}. For a function $u\in Y_{a,b}(\R^n)$, we define
\[
\delta_{CKN}(u) = \frac{\Big(\irn \frac{|\Delta u|^2}{|x|^{2a}} dx \Big)^{\frac12}\Big(\irn \frac{|\nabla u|^2}{|x|^{2b}} dx\Big)^{\frac12}}{\irn \frac{|\nabla u|^2}{|x|^{a+b+1}} dx} -C_2(n,a,b)
\]
if $u\not\equiv 0$ and $\delta_{CKN}( u) =0$ if $u \equiv 0$. Our next result reads as follows.
\begin{theorem}\label{MainIII}
Let $n\geq 2$ and $a, b$ be real numbers such that $1-b + a\not=0$, $n-2a -4\not=0$ and $(n-2a -2)^2 > 8(1+a)(1+2a)$. Then,  for any function $u\in Y_{a,b}(\R^n)$, we have
\begin{equation*}
\delta_{CKN}(u) \geq C \inf\left\{\frac{\irn \frac{|\nabla u - \nabla v|^2}{|x|^{a+b+1}} dx}{\irn \frac{|\nabla u |^2}{|x|^{a+b+1}} dx}\, :\, v \in \mathcal N_{a,b},\quad \irn \frac{|\nabla u |^2}{|x|^{a+b+1}} dx = \irn \frac{|\nabla v|^2}{|x|^{a+b+1}} dx\right\}
\end{equation*}
for some positive constant $C$ depending only on $n,a,b$.
\end{theorem}
We next provide some direct consequences of Theorem \ref{MainIII} in some special cases of parameters. 	Let $b= -a -1$, $a > -1$ and $(n-2a -2)^2 > 8(1+a)(1+2a)$, Theorem \ref{MainIII} gives the stability estimate for the weighted second order Heisenberg-Pauli-Weyl uncertainty principle 
 \begin{corollary}
Let $b= -a -1$, $a > -1$ and $(n-2a -2)^2 > 8(1+a)(1+2a)$. For $u\in Y_{a,-a-1}(\R^n)$, we have
\begin{equation}\label{eq:stabweightedHUP}
\delta_{CKN}(u) \geq C \inf\left\{\frac{\irn |\nabla u - \nabla v|^2 dx}{\irn |\nabla u |^2 dx}\, :\, v \in \mathcal N_{a,-a-1},\quad \irn |\nabla u |^2 dx = \irn |\nabla v|^2dx\right\}.
\end{equation}
\end{corollary}
\noindent In the special case $a =0$, we recover the stability result in \cite{DNCVPDE} from \eqref{eq:stabweightedHUP} when $n\geq 5$. 

	Let $b = -a$, $a > -\frac12$ and $(n-2a -2)^2 > 8(1+a)(1+2a)$, we get from Theorem \ref{MainIII} the stability estimate for the weighted second order Hydrogen uncertainty principle
\begin{corollary}
Let $b = -a$, $a > -\frac12$ and $(n-2a -2)^2 > 8(1+a)(1+2a)$. For $u\in Y_{a,-a}(\R^n)$, we have
\begin{equation}\label{eq:Stab2ndHyUP}
\delta_{CKN}(u) \geq C \inf\left\{\frac{\irn \frac{|\nabla u - \nabla v|^2}{|x|} dx}{\irn \frac{|\nabla u |^2}{|x|} dx}\, :\, v \in \mathcal N_{a,-a},\quad \irn \frac{|\nabla u |^2}{|x|} dx = \irn \frac{|\nabla v|^2}{|x|} dx\right\}.
\end{equation}
\end{corollary}
\noindent In the special case $a =0$, the estimate \eqref{eq:Stab2ndHyUP} gives us a stability result for the second order Hydrogen uncertainty principle when $n\geq 5$.

Finally, let us give some comments on our approach to prove the main results in this paper. The inequality \eqref{eq:2ndCKN} in the particular case $b = -a -1$ was proved in \cite{DNAFM} by using the factorization method and in \cite{CFL22} by using the {\it expanding-the square} method. These methods are usually used to prove the functional inequalities such as Hardy and CKN inequalities. However,  it is difficult  to apply these methods to the inequality \eqref{eq:2ndCKN} for general parameters $a, b$. Instead, our approach in this paper is to apply the spherical harmonic decomposition method to prove Theorem \ref{MainII} (i.e, to prove the inequality \eqref{eq:2ndCKN} for general parameters $a$ and $b$). It is well-known that the spherical harmonic decomposition method is a powerful mathematical tool in studying the functional inequalities, especially in the sharp forms and their stability versions (see, e.g, \cite{TerZoAiM,DNCVPDE,DLL26,CFL22} and the references therein). It enables us to reduce the proof of \eqref{eq:2ndCKNnew} for general functions  to that  for radial functions (on each node). This  makes the problem more flexible to handle. More precisely, by using this method, to prove {Theorem \ref{MainII}},  it is enough to establish  the following one-dimensional inequalities 
\begin{align*}
\Big(\int_0^\infty& (u_k'')^2r^{n-2a -1} dr+ (n-1)(1+2a) \int_0^\infty (u_k')^2 r^{n-2a -3} dr\Big) \int_0^\infty (u_k')^2 r^{n-2b -1} dr\notag\\
&\geq C_2(n,a,b)^2 \Big(\int_0^\infty (u_k')^2 r^{n - a -b -2} dr\Big)^2.
\end{align*}
for any $k\geq 0$ and
\begin{align*}
&\Big(2\int_0^\infty (u_k')^2 r^{n-2a -3} dr + (c_k + 2(1+a)(n-2a -4))\int_0^\infty u_k^2 r^{n -2a -5} dr\Big) \times \notag\\
&\qquad\qquad\qquad \times \int_0^\infty (u_k)^2 r^{n-2a -3} dr\notag\\
&\geq C_2(n,a,b)\Big(\int_0^\infty u_k^2 r^{n -a -b-4} dr\Big)^2
\end{align*}
for any $k\geq 1$ which are consequences of Lemma \ref{modelineq} and the one dimensional-Hardy inequality \eqref{eq:1DHardy} (see Section \ref{s2}). In order to prove Theorem \ref{MainI} and Theorem \ref{MainIII}, we follows the approach in \cite{DNCVPDE}, where we established the stability version of the second order Heisenberg-Pauli-Weyl uncertainty principle. Employing the spherical harmonic decomposition method and Lemma \ref{modelineq}, we will prove that the sharp constants in the weighted $L^2-$CKN inequalities for curl-free vector fields and second order derivatives can be improved by larger constants when restricted to curl-free vector fields $U = \nabla u\in X_{a,b}(\R^n)$ or to functions $u\in Y_{a,b}(\R^n)$ which are  orthogonal to all radial functions (see Lemma \ref{Improved2ndvf} and Lemma \ref{Improved2ndCKN1} below). At this step, the conditions $(n-2a)^2 > 4(n+1)$ and $(n-2a-2)^2 > 8(1+a)(1+2a)$ play the important role. These improvements imply that if $\delta(U)$ (or $\delta_{CKN}(u)$) is small then $U$ (or $\nabla u$) is close to a curl-free vector field { $V =\nabla v$ (or $\nabla v$)} in a weighted $L^2$ space on $\R^n$, where $v$ is a radial function. Using Lemma \ref{stabmodel}, we obtain the stability results for such curl-free vector field $V$ (or function $v$). Theorem \ref{MainI} and Theorem \ref{MainIII} directly follow from these two estimates.

The rest of this paper is organized as follows. In Section \ref{s2}, we prepare some ingredients which are used to prove our main results such as the spherical harmonic decomposition method, some integral identities, the one-dimension integral inequalities and their improvements. Section \ref{s3} is devoted to prove Theorem \ref{MainII}. In Section \ref{s4}, we prove Theorem \ref{MainI} on the stability version of the weighted $L^2-$CKN inequality for curl-free vector fields. In  Section \ref{s5}, we give the proof of Theorem \ref{MainIII} on the stability version of the weighted second order $L^2-$CKN inequality.

\section{Preliminaries}\label{s2}
In this section, we prepare some ingredients for  the proof of our main Theorems. Let $U\in X_{a,b}(\R^n)$ and $u\in C^\infty(\R^n\setminus\{0\})$ be a scalar potential of $U$, i.e, $\nabla u = U$. We then have
\begin{equation}\label{eq:condu1}
\lim_{|x|\to 0,\infty} |x|^{-1 + \frac n2 - a}\nabla u(x) = 0
\end{equation}
and
\begin{equation*}
\lim_{|x|\to 0,\infty} |x|^{\frac n2 - b} \nabla u(x) = 0.
\end{equation*}
We first establish  the following identity.
\begin{lemma}\label{D2toDelta}
It holds that
\begin{align}\label{eq:D2toDelta}
\int_{\R^n} \frac{|\nabla U|^2}{|x|^{2a}} dx &=\irn \frac{(\Delta u)^2}{ |x|^{2a}} dx -2a(n -2a -3)\irn \frac{|\nabla u|^2}{|x|^{2 + 2a}} dx\notag\\
&\qquad -4a(1+a) \irn \frac{\Big(\nabla u \cdot \frac x{|x|}\Big)^2 }{|x|^{2 + 2a}} dx.
\end{align}
\end{lemma} 

\begin{proof}
Notice that 
\[
\int_{\R^n} \frac{|\nabla U|^2}{|x|^{2a}} dx = \int_{\R^n} \frac{|\nabla^2 u|^2}{|x|^{2a}} dx
\]
where $\nabla^2 u = (\partial_{ij}^2 u)_{n\times n}$ denotes the Hessian matrix of $u$ and $|\nabla^2 u|^2 = \sum_{i,j=1}^n (\partial^2_{ij}u)^2$. For $\ep < R$, we denote by $A_{\ep,R} = \{x\in \R^n\,: \, \ep \leq |x| \leq R\}$, and $\nu(x)$ the outward normal vector to $\pa A_{\ep,R}$. Let $S$ denote the $(n-1)$ dimensional surface area measure in $\R^n$. Using an integration by parts, we have
\begin{align}\label{eth1}
&	\int_{A_{\ep,R}} \frac{|\nabla^2 u|^2}{|x|^{2a}} dx= \sum_{i=1}^n \int_{A_{\ep,R}} |\nabla \partial_iu|^2 |x|^{-2a} dx\\ \notag
& = \sum_{i=1}^n \Big(-\int_{A_{\ep,R}} \partial_i u \Delta \partial_i u |x|^{-2a} dx + 2a \int_{A_{\ep,R}} \partial_i u \nabla\partial_i u \cdot x |x|^{-2a -2} dx\\\notag
&\qquad + \int_{\partial A_{\ep,R}} \partial_i u \nabla\partial_i u \cdot \nu(x) |x|^{-2a} dS\Big)\\\notag
&=-\int_{A_{\ep,R}} \nabla u \cdot \nabla\Delta u |x|^{-2a} dx  + a \int_{A_{\ep,R}} \nabla(|\nabla u|^2) \cdot x |x|^{-2a -2} dx\\\notag
&\qquad + \int_{\partial A_{\ep,R}} \nabla u \cdot \nabla^2 u(\nu(x)) |x|^{-2a} dS\\\notag
&=\int_{A_{\ep,R}} (\Delta u)^2 |x|^{-2a} dx - 2a \int_{A_{\ep,R}} \Delta u \nabla u\cdot x|x|^{-2a -2} dx\\\notag
&\qquad -a(n -2a-2) \irn \frac{|\nabla u|^2}{|x|^{2a +2} } dx  -\int_{\partial A_{\ep,R}} \Delta u\, \nabla u \cdot \nu(x) |x|^{-2a} dS\\ \notag
&\qquad + a\int_{\partial A_{\ep,R}} |\nabla u|^2 x \cdot \nu(x) |x|^{-2a-2} dS + \int_{\partial A_{\ep,R}} \nabla u \cdot \nabla^2 u(\nu(x)) |x|^{-2a} dS.
\end{align}
Using again an integration by parts, we also have
\begin{align}\label{eth2}
\begin{split}
	\int_{A_{\ep,R}} &\Delta u \nabla u\cdot x|x|^{-2a -2} dx \\
&= -\int_{A_{\ep,R}} \nabla u \cdot \nabla(\nabla u\cdot x |x|^{-2a -2}) dx\\
	&\qquad + \int_{\partial A_{\ep,R}} \nabla u \cdot \nu(x) \nabla u\cdot x |x|^{-2a -2} dS\\
&= -\int_{A_{\ep,R}} x \cdot \nabla^2u (\nabla u) |x|^{-2a -2} dx - \int_{A_{\ep,R}} |\nabla u|^2 |x|^{-2a -2} dx\\
&\qquad + (2a +2) \int_{A_{\ep,R}} (\nabla u\cdot x)^2 |x|^{-2a -4} dx + \int_{\partial A_{\ep,R}} \nabla u \cdot \nu(x) \nabla u\cdot x |x|^{-2a -2} dS\\
&=-\frac12 \int_{A_{\ep,R}} \nabla(|\nabla u|^2) \cdot x |x|^{-2a -2} dx - \int_{A_{\ep,R}}|\nabla u|^2 |x|^{-2a -2} dx\\
&\qquad + (2a +2) \int_{A_{\ep,R}} (\nabla u\cdot x)^2 |x|^{-2a -4} dx+ \int_{\partial A_{\ep,R}} \nabla u \cdot \nu(x) \nabla u\cdot x |x|^{-2a -2} dS\\
&= \frac{n-2a -2}2 \int_{A_{\ep,R}} |\nabla u|^2 |x|^{-2a -2} dx -\frac12\int_{\partial A_{\ep,R}} |\nabla u|^2 x \cdot \nu(x) |x|^{-2a-2} dS\\
&\qquad - \int_{A_{\ep,R}} |\nabla u|^2 |x|^{-2a -2} dx + (2a +2) \int_{A_{\ep,R}} (\nabla u\cdot x)^2 |x|^{-2a -4} dx\\
&\qquad + \int_{\pa A_{\ep,R}} \nabla u \cdot \nu(x) \nabla u\cdot x |x|^{-2a -2} dS\\
& = \frac{n-2a -4}2 \int_{A_{\ep,R}} \frac{|\nabla u|^2}{ |x|^{2a +2}} dx + (2a +2) \int_{A_{\ep,R}} \frac{\Big(\nabla u\cdot \frac{x}{|x|}\Big)^2}{ |x|^{2a +2}} dx\\
&\qquad -\frac12\int_{\partial A_{\ep,R}} |\nabla u|^2 x \cdot \nu(x) |x|^{-2a-2} dS + \int_{\pa A_{\ep,R}} \nabla u \cdot \nu(x) \nabla u\cdot x |x|^{-2a -2} dS.
\end{split}
\end{align}
We next treat the integral on the boundary of $A_{\ep,R}$. Notice by using the spherical coordinate that
\[
\int_0^\infty \Big(\int_{\{|x| =r\}}|\nabla^2 u|^2 dS\Big) r^{-2a} dr =  \irn \frac{|\nabla^2 u|^2}{|x|^{2a}} dx < \infty.
\]
Consequently, there exist two sequence $\{\ep_i\}_i$ and $\{R_i\}_i$,  $\ep_i \to 0^+$ and $R_i\to \infty$ as $i\to \infty$, such that
\[
\lim_{i \to \infty} \ep_i^{1-2a} \int_{\{|x| =\ep_i\}}|\nabla^2 u|^2 dS = \lim_{i\to \infty} R_i^{1-2a}\int_{\{|x| =R_i\}}|\nabla^2 u|^2 dS = 0,
\]
which is equivalent to

\begin{equation}\label{etuanhoang1}
	\lim_{i\to \infty} \int_{\pa A_{\ep_i,R_i}} |\nabla^2 u|^2 |x|^{1-2a} dS =0.
\end{equation}
On the other hand, from  \eqref{eq:condu1}, we have
\begin{equation}\label{etuanhoang2}
	\lim_{i\to \infty} \int_{\pa A_{\ep_i,R_i}} |\nabla u|^2 |x|^{-1-2a} dS =0.
\end{equation}
In addition, 
\[
\Big|\int_{\partial A_{\ep_i,R_i}} \Delta u\, \nabla u \cdot \nu(x) |x|^{-2a} dx\Big| \leq \frac n2 \int_{\pa A_{\ep_i,R_i}} |\nabla^2 u|^2 |x|^{1-2a} dS + \frac 12 \int_{\pa A_{\ep_i,R_i}} |\nabla u|^2 |x|^{-1-2a} dS,
\]
\[
\Big|\int_{\partial A_{\ep,R}} \nabla u \cdot \nabla^2 u(\nu(x)) |x|^{-2a} dx\Big| \leq \frac12 \int_{\pa A_{\ep_i,R_i}} |\nabla^2 u|^2 |x|^{1-2a} dS + \frac 12 \int_{\pa A_{\ep_i,R_i}} |\nabla u|^2 |x|^{-1-2a} dS,
\]
\[
\Big|\int_{\partial A_{\ep,R}} |\nabla u|^2 x \cdot \nu(x) |x|^{-2a-2} dx\Big| \leq \int_{\partial A_{\ep,R}} |\nabla u|^2  |x|^{-2a-1} dS,
\]
and
\[
\Big|\int_{\pa A_{\ep,R}} \nabla u \cdot \nu(x) \nabla u\cdot x |x|^{-2a -2} dx\Big| \leq \int_{\partial A_{\ep,R}} |\nabla u|^2  |x|^{-2a-1} dS.
\]
It follows from these estimates and \eqref{etuanhoang1}, \eqref{etuanhoang2} that  all integrals on the boundary $\pa A_{\ep_i,R_i}$ tend to $0$ as $i\to \infty$. Inserting \eqref{eth2} into \eqref{eth1} with $\epsilon =\epsilon_i$ and $R =R_i$, and then letting $i\to \infty$, we arrive at \eqref{eq:D2toDelta}. The proof of Lemma \ref{D2toDelta} is completed.
\end{proof}

As mentioned in the introduction, one of the main ingredients in our approach is the spherical harmonic decomposition method. In the following, we shall list some facts on this technique. For  $n \geq 2$, and $k \geq 0$ is an integer, we denote by $\mathcal H_k$  the eigenspace of the spherical Laplace operator $-\Delta_{S^{n-1}}$ on the sphere $S^{n-1}$ with respect to the eigenvalue $c_k = k (n+k -2)$, that is $-\Delta_{S^{n-1}} \phi = c_k \phi$ with $\phi \in \mathcal H_k$. For each $k$, let $\phi_{kl}$, $l =1,2,\ldots,\text{ dim }\mathcal H_k$ be an orthonormal basis of $\mathcal H_k$ in $L^2(S^{n-1})$ with standard surface area measure on $S^{n-1}$, i.e,
\[
\int_{S^{n-1}} \phi_{kl}(\omega)\phi_{kl'}(\omega) d\omega = \begin{cases}
1 &\mbox{if $l = l'$},\\
0 &\mbox{if $l\not= l'$}.
\end{cases}
\]
It is well-known that such the functions $\phi_{kl}, k =0,1,2, \ldots$ and $l =1,2,\dots,\text{ dim }\mathcal H_k$ constitute an orthonormal basis of $L^2(S^{n-1})$. A function $u\in C^\infty(\R^n\setminus\{0\})$ can be decomposed as
\begin{equation}\label{eq:decompu}
u(x) = u(r\omega) = \sum_{k=0}^\infty \sum_{l=1}^{\text{ dim } \mathcal H_k} u_{kl}(r) \phi_{kl}(\omega),\quad r = |x|,\, \omega = \frac{x}{|x|}.
\end{equation}
In fact $u_{kl}(r) = \int_{S^{n-1}} u(r\omega) \phi_{kl}(\omega) d\omega \in C^\infty(\R^n \setminus\{0\})$ is a radial function. In the sequel, for simplicity of the notations, we will write the decomposition \eqref{eq:decompu} as
\[
u(x) = u(r\omega) = \sum_{k=0}^\infty  u_{k}(r) \phi_{k}(\omega),
\]
where $-\Delta_{S^{n-1}} \phi_k = c_k \phi_k$.

Suppose that $U =\nabla u\in X_{a,b}(\R^n)$ or $u\in Y_{a,b}(\R^n)$. As discussed above, we can decompose $u$ as
\[
u(x) = u(r\omega) = \sum_{k=0}^\infty  u_{k}(r) \phi_{k}(\omega), \quad r = |x|,\, \omega = \frac x{|x|}.
\]
It implies that 
\[
\nabla u (x) = \sum_{k=0}^\infty\Big( u_k'(r) \phi_k(\omega) \omega + \frac1{r} u_k(r) \nabla_{S^{n-1}} \phi_k(\omega)\Big),
\]
where $\nabla_{S^{n-1}}$ denotes the spherical gradient on $S^{n-1}$. By the orthogonality and \eqref{eq:condu1}, we  have 
\begin{equation}\label{eq:daohamukbehavior}
\lim_{r \to 0, \infty} r^{-1-a +\frac n2}u_k'(r) =0,\quad k =0,1,2,\ldots,
\end{equation}
and
\begin{equation}\label{eq:ukbehavior}
\lim_{r \to 0, \infty} r^{-2-a +\frac n2}u_k(r) =0,\quad k =1,2,\ldots.
\end{equation}
It remains to consider the function $u_0$. Firstly, consider  $n -2a -4 >0$. From  \eqref{eq:daohamukbehavior},  it holds for $r< s$ that
\[
|u_0(r) -  u_0(s)| = \Big|\int_r^s u_0'(t) dt\Big| \leq C\int_r^s t^{1 +a -\frac n2} dt \leq \frac{2C}{n -2a -4} r^{2+a -\frac n2} \to 0,
\]
as $r\to \infty$. Thus, there exists $\lim_{r\to \infty} u_0(r) = c$, and
\[
u_0(r) - c = -\int_r^\infty u_0'(t) dt.
\]
Again by using \eqref{eq:daohamukbehavior}, it is not hard to prove that
\[
\lim_{r\to 0, \infty} (u_0(r) - c) r^{\frac n2 -a - 2} =0.
\]
The case $n - 2a -4 < 0$ is treated similarly. Hence,
there is a constant $c$ such that
\begin{equation}\label{eq:u01}
\lim_{r\to 0, \infty} (u_0(r) - c) r^{\frac n2 -a - 2} =0
\end{equation} 
if $n-2a -4 \not =0$. 
To summarize, if $U = \nabla u \in X_{a,b}(\R^n)$ or $u \in Y_{a,b}(\R^n)$ with $n-2a-4\not=0$, by replacing $u$ by $u -c$ if necessary, we can assume that the functions $u_k$ in the spherical harmonic decomposition of $u$ satisfy the estimates \eqref{eq:daohamukbehavior}, \eqref{eq:ukbehavior},  and \eqref{eq:u01} (with $c =0$). 

On one hand,
\[
\Delta u(x) = \sum_{k=0}^\infty \Big(u_k''(r) + \frac{n-1}r u_k'(r) - \frac{c_k}{r^2} u_k(r)\Big)\phi_k.
\]
On the other hand, using an  integration by parts and the behavior of $u_k$ near $0$ and infinity, it is easy to see that for any $\alpha, \beta \in \R$, we have
\begin{align*}
\int_0^\infty &\Big(u_k''(r) + \frac{\alpha}r u_k'(r) - \frac{\beta}{r^2} u_k(r)\Big)^2 r^{n-2a -1} dr\notag\\
&= \int_0^\infty (u_k'')^2 r^{n-2a -1} dr + (\alpha^2 - \alpha(n-2a -2) +2\beta) \int_0^\infty (u_k')^2 r^{n-2a -3} dr\notag\\
&\qquad + (\beta^2 +\alpha\beta(n-2a -4) - \beta(n-2a -3)(n-2a -4))\int_0^\infty u_k^2 r^{n -2a -5} dr.
\end{align*}
Then, we deduce that 
\begin{align}\label{eq:integralDelta}
\irn \frac{(\Delta u)^2}{ |x|^{2a}} dx &= n\sigma_n\sum_{k=0}^n \int_0^\infty \Big(u_k''(r) + \frac{n-1}r u_k'(r) - \frac{c_k}{r^2} u_k(r)\Big)^2 r^{n-2a -1} dr\notag\\
&=n\sigma_n \sum_{k=1}^\infty\Bigg(\int_0^\infty (u_k'')^2 r^{n-2a -1} dr + ((n-1)(1+2a) +2c_k) \int_0^\infty (u_k')^2 r^{n-2a -3} dr\notag\\
&\qquad\qquad\qquad + c_k(c_k + 2(1+a)(n-2a -4))\int_0^\infty u_k^2 r^{n -2a -5} dr\Bigg)
\end{align}
and
\begin{align*}
	\irn \frac{|\nabla u|^2}{|x|^{2 + 2a}} dx& =n\sigma_n \sum_{k=0}^\infty\Bigg(\int_0^\infty (u_k')^2 r^{n-2a -3} dr + c_k \int_0^\infty u_k^2 r^{n -2a -5} dr\Bigg)\\
	\irn \frac{\Big(\nabla u \cdot \frac x{|x|}\Big)^2 }{|x|^{2 + 2a}} dx &= n\sigma_n \sum_{k=0}^\infty \int_0^\infty (u_k')^2 r^{n-2a -3} dr,
\end{align*}
where $\sigma_n$ is the volume of the unit ball in $\R^n$.

Plugging these equalities into \eqref{eq:D2toDelta}, we obtain
\begin{align}\label{eq:integralnabla2}
\int_{\R^n} \frac{|\nabla U|^2}{|x|^{2a}} dx& = n\sigma_n \sum_{k=0}^\infty\Bigg(\int_0^\infty (u_k'')^2 r^{n-2a -1} dr + (n-1 +2c_k) \int_0^\infty (u_k')^2 r^{n-2a -3} dr\notag\\
&\qquad\qquad\qquad + c_k(c_k + 2(n-3a -4))\int_0^\infty u_k^2 r^{n -2a -5} dr\Bigg).
\end{align}  
Moreover, we also have
\begin{equation}\label{eq:nablab}
\irn \frac{|U|^2}{|x|^{2b}} dx =\irn \frac{|\nabla u|^2}{|x|^{2b}} dx=n\sigma_n \sum_{k=0}^\infty\Bigg(\int_0^\infty (u_k')^2 r^{n-2b -1} dr + c_k \int_0^\infty u_k^2 r^{n -2b -3} dr\Bigg)
\end{equation}
and
\begin{equation}\label{eq:nablaab}
\irn \frac{|U|^2}{|x|^{a+b+1}} dx =\irn \frac{|\nabla u|^2}{|x|^{2b}} dx =n\sigma_n \sum_{k=0}^\infty\Bigg(\int_0^\infty (u_k')^2 r^{n-a -b-2} dr + c_k \int_0^\infty u_k^2 r^{n -a-b -4} dr\Bigg).
\end{equation}

We are now ready to establish a key lemma.
\begin{lemma}\label{modelineq}
Let $A, B, C \in \R$ with $4C + (A-2)^2 >0$. Then for any $f \in C_0^\infty((0,\infty))$ such that
\[
\int_0^\infty (f')^2 r^{A -1} dr < \infty,\qquad \int_0^\infty f^2 r^{A -3} dr < \infty, \qquad \int_0^\infty f^2 r^{B-1} dr < \infty,
\]
and 
\[
\lim_{r\to 0, \infty} f(r) r^{\frac A2 -1} = 0,
\]
we have
\begin{description}
\item(i) If $B-A+2 > 0$ then
\begin{align}\label{eq:model1}
\Big(\int_0^\infty (f')^2& r^{A -1} dr + C \int_0^\infty f^2 r^{A -3} dr\Big) \int_0^\infty f^2 r^{B-1} dr\notag\\
&\geq  \Big(\frac{B-A+2 + 2\sqrt{(A-2)^2 + 4C}}4\Big)^2 \Big(\int_0^\infty f^2 r^{\frac{A+B}2-2} dr\Big)^2.
\end{align}
Furthermore, the equality holds when 
\begin{equation}\label{eq:extre1}
f(r)=f_0(r) = r^{-\frac{A-2 - \sqrt{(A-2)^2 + 4C}}2} e^{-\frac{2r^{\frac{B-A+2}2}}{B-A+2}}.
\end{equation}
\item(ii) If $B-A+2 < 0$ then
\begin{align}\label{eq:model2}
\Big(\int_0^\infty (f')^2& r^{A -1} dr + C \int_0^\infty f^2 r^{A -3} dr\Big) \int_0^\infty f^2 r^{B-1} dr\notag\\
&\geq  \Big(\frac{B-A+2 - 2\sqrt{(A-2)^2 + 4C}}4\Big)^2 \Big(\int_0^\infty f^2 r^{\frac{A+B}2-2} dr\Big)^2.
\end{align}
Furthermore, the equality holds when 
\begin{equation}\label{eq:extre2}
f(r)=\bar{f_0}(r) = r^{-\frac{A-2 +\sqrt{(A-2)^2 + 4C}}2} e^{-\frac{2r^{\frac{B-A+2}2}}{B-A+2}}.
\end{equation}
\end{description}
\end{lemma}
\begin{proof}
We first prove part $(i)$. Let 
\[
\gamma = \frac{A-2 - \sqrt{(A-2)^2 + 4C}}2.
\]
Using an integration by parts and the behaviors of $f$ at $0$ and infinity, we get
\begin{equation}\label{eq:idgamma}
\int_0^\infty\Big(f' + \frac{\gamma}r f\Big)^2 r^{A-1} dr = \int_0^\infty (f')^2 r^{A -1} dr + C \int_0^\infty f^2 r^{A -3} dr.
\end{equation}
Since 
\[
\int_0^\infty f(r)^2 r^{B-1} dr < \infty,
\]
 there exist two sequences $\{\epsilon_i\}_i$ and $\{R_i\}_i$ such that $\epsilon_i \to 0$ and $R_i \to \infty$ as $i\to \infty$, and
\[
\lim_{i\to \infty} f(\ep_i) \ep_i^{\frac B2} = \lim_{i\to \infty} f(R_i) R_i^{\frac B2} = 0.
\]
By an  integration by parts,  we obtain
\begin{align}\label{eq:IbPf}
\Big(\frac{A+B}2 -1\Big)\int_{\ep_i}^{R_i} f^2 r^{\frac{A+B}2-2} dr& = \int_{\ep_i}^{R_i} f^2 \Big(r^{\frac{A+B}2-1}\Big)' dr\notag\\
& = -2\int_{\ep_i}^{R_i} f' f r^{\frac{A+B}2-1} dr + \Big(f(R_i)^2 R_i^{\frac{A+B}2-1} -f(\ep_i)^2 \ep_i^{\frac{A+B}2-1}\Big)\notag\\
& = -2 \int_{\ep_i}^{R_i}\Big(f' + \frac{\gamma}r f\Big) f r^{\frac{A+B}2-1} dr + 2\gamma \int_{\ep_i}^{R_i} f^2 r^{\frac{A+B}2-2} dr\notag\\
&\qquad + \Big(f(R_i)^2 R_i^{\frac{A+B}2-1} -f(\ep_i)^2 \ep_i^{\frac{A+B}2-1}\Big).
\end{align}
Letting $i\to \infty$, we have
\begin{equation}\label{eq:keyid}
\Big(\frac{A+B-2}4 - \gamma\Big)\int_0^\infty f^2 r^{\frac{A+B}2-2} dr = -\int_0^\infty \Big(f' + \frac{\gamma}r f\Big) f r^{\frac{A+B}2-1} dr.
\end{equation}
The inequality \eqref{eq:model1} is followed from \eqref{eq:keyid} by using H\"older's  inequality, the identity \eqref{eq:idgamma} and the fact
\[
\frac{A+B-2}4 - \gamma = \frac{B-A+2 + 2\sqrt{(A-2)^2 + 4C}}4.
\]
For the sharpness of \eqref{eq:model1}, it easy to see that if 
\[
f(r) = r^{-\frac{A-2 - \sqrt{(A-2)^2 + 4C}}2} e^{-\frac{2r^{\frac{B-A+2}2}}{B-A+2}},
\]
then all integrals in \eqref{eq:model1} are finite  and 
\[
\lim_{r\to 0} f(r)r^{\frac{A}2 -1} = \lim_{r\to \infty} f(r) r^{\frac{A}2 -1} = 0.
\]
Hence, the equality \eqref{eq:IbPf} holds for this function $f$. Moreover, the function $f$ satisfies
\[
\Big(f'(r) + \frac{\gamma}r f(r)\Big) r^{\frac{A}2 -1} = -f(r) r^{\frac B2 -1},
\]
then the equality holds when applying H\"older inequality to \eqref{eq:keyid}. This proves the sharpness of \eqref{eq:model1}.

To prove part $(ii)$, we make the change of function $g(r) = f(1/r)$. By the simple computations, it holds
\[
\int_0^\infty (f'(r))^2 r^{A-1} dr = \int_0^\infty (g'(r))^2 r^{3-A} dr,
\]
\[
\int_0^\infty f(r)^2 r^{A-3} dr = \int_0^\infty g(r)^2 r^{1-A} dr,
\]
\[
\int_0^\infty f(r)^2 r^{B-1} dr = \int_0^\infty g(r)^2 r^{-B-1} dr,
\]
and
\[
\int_0^\infty f(r)^2 r^{\frac{A+B}2-2} dr = \int_0^\infty g(r)^2 r^{-\frac{A+B}2} dr.
\]
Hence the inequality \eqref{eq:model2} follows from part $(i)$ with parameters $4-A$, $-B$ and $C$.
\end{proof}

In the next Lemma, we establish the stability versions for inequalities from Lemma \ref{modelineq}.
\begin{lemma}\label{stabmodel}
Given $A,B,C\in \R$ such that $(A-2)^2 + 4C >0$. Then for any $f \in C_0^\infty((0,\infty))$ satisfying
\[
\int_0^\infty (f')^2 r^{A -1} dr < \infty,\qquad \int_0^\infty f^2 r^{A -3} dr < \infty, \qquad \int_0^\infty f^2 r^{B-1} dr < \infty,
\]
and 
\[
\lim_{r\to 0, \infty} f(r) r^{\frac A2 -1} = 0,
\]
we have
\begin{description}
\item(i) If $B-A+2 > 0$, then 
\begin{align*}
&\frac{\Big(\int_0^\infty (f')^2 r^{A -1} dr + C \int_0^\infty f^2 r^{A -3} dr\Big)^{\frac12}\Big( \int_0^\infty f^2 r^{B-1} dr\Big)^{\frac12}}{\int_0^\infty f^2 r^{\frac{A+B}2-2} dr}\notag\\
&\qquad\qquad\qquad \qquad\qquad\qquad - \frac{B-A+2 + 2\sqrt{(A-2)^2 + 4C}}4 \notag\\
&\geq \frac{B-A+2}2 \inf_{c\in \R, \lambda > 0}  \frac{\int_0^\infty |f(r) - c f_0(\lambda r)|^2 r^{\frac{A+B}2-2} dr}{\int_0^\infty f^2 r^{\frac{A+B}2-2} dr},
\end{align*}
where $f_0$ is given by \eqref{eq:extre1}.
\item(ii) If $B-A+2 < 0$, then 
\begin{align*}
&\frac{\Big(\int_0^\infty (f')^2 r^{A -1} dr + C \int_0^\infty f^2 r^{A -3} dr\Big)^{\frac12}\Big( \int_0^\infty f^2 r^{B-1} dr\Big)^{\frac12} }{\int_0^\infty f^2 r^{\frac{A+B}2-2} dr}\notag\\
&\qquad\qquad\qquad \qquad\qquad\qquad - \frac{-B+A-2 + 2\sqrt{(A-2)^2 + 4C}}4 \notag\\
&\geq \frac{-B+A -2}2 \inf_{c\in \R, \lambda > 0}  \frac{\int_0^\infty |f(r) - c \bar{f_0}(\lambda r)|^2 r^{\frac{A+B}2-2} dr}{\int_0^\infty f^2 r^{\frac{A+B}2-2} dr},
\end{align*}
where $\bar{f_0}$ is given by \eqref{eq:extre2}.
\end{description}
\end{lemma}
\begin{proof}
As in the proof of Lemma \ref{modelineq}, it is enough to prove the case $A-B+2 >0$. The remain case is proved by changing the function. From  the homogeneity and scaling invariance, we may assume that $\int_0^\infty f^2 r^{\frac{A+B}2-2} dr =1$ and 
\[
\int_0^\infty (f')^2 r^{A -1} dr + C \int_0^\infty f^2 r^{A -3} dr =\int_0^\infty f^2 r^{B-1} dr.
\]
These assumptions together with \eqref{eq:idgamma} and \eqref{eq:keyid} imply
\begin{align*}
&\frac{\Big(\int_0^\infty (f')^2 r^{A -1} dr + C \int_0^\infty f^2 r^{A -3} dr\Big)^{\frac12}\Big( \int_0^\infty f^2 r^{B-1} dr\Big)^{\frac12}}{\int_0^\infty f^2 r^{\frac{A+B}2-2} dr}\\
&=\frac12 \Big(\int_0^\infty (f')^2 r^{A -1} dr + C \int_0^\infty f^2 r^{A -3} dr\Big) + \frac12 \int_0^\infty f^2 r^{B-1} dr\\
& =\frac12\int_0^\infty\Big(f' + \frac{\gamma}r f\Big)^2 r^{A-1} dr + \frac12 \int_0^\infty f^2 r^{B-1} dr\\
& = \frac12\int_0^\infty\Big(f' + \frac{\gamma}r f + \frac{f}{r^{\frac{A-B}2}}\Big)^2 r^{A-1} dr - \int_0^\infty\Big(f' + \frac{\gamma}r f\Big)^2 f r^{\frac{A+B}2-1} dr\\
&=\frac12\int_0^\infty\Big(f' + \frac{\gamma}r f + \frac{f}{r^{\frac{A-B}2}}\Big)^2 r^{A-1} dr + \frac{B-A+2 + 2\sqrt{(A-2)^2 + 4C}}4.
\end{align*}
Therefore, we obtain
\begin{align}\label{eq:idkey1}
&\frac{\Big(\int_0^\infty (f')^2 r^{A -1} dr + C \int_0^\infty f^2 r^{A -3} dr\Big)^{\frac12}\Big( \int_0^\infty f^2 r^{B-1} dr\Big)^{\frac12}}{\int_0^\infty f^2 r^{\frac{A+B}2-2} dr}\notag\\
&\qquad -  \frac{B-A+2 + 2\sqrt{(A-2)^2 + 4C}}4\notag\\
&=\frac12\int_0^\infty\Big(f' + \frac{\gamma}r f + \frac{f}{r^{\frac{A-B}2}}\Big)^2 r^{A-1} dr.
\end{align}
Making the change of function
\[
f(r) = g(r) r^{-\gamma} e^{-\frac{r^{(B-A)/2 +1}}{(B-A)/2 +1}},
\]
we have
\[
f'(r) = g'(r) r^{-\gamma} e^{-\frac{r^{(B-A)/2 +1}}{(B-A)/2 +1}} -\frac\gamma r f(r) - r^{\frac{B-A}2} f(r).
\]
Hence, 
\begin{equation}\label{eq:idkey2}
\int_0^\infty\Big(f' + \frac{\gamma}r f + \frac{f}{r^{\frac{A-B}2}}\Big)^2 r^{A-1} dr = \int_0^\infty(g'(r))^2 r^{1 + \sqrt{(A-2)^2 + 4 C}} e^{-2\frac{r^{(B-A)/2 +1}}{(B-A)/2 +1}} dr.
\end{equation} 
Making the change of variable 
\[
\frac{s^2}4 = \frac{r^{\frac{A-B+2}2}}{\frac{A-B+2}2}
\]
and letting $h(s) = g(r)$, we get
\begin{align}\label{eq:stabf1}
&\int_0^\infty(g'(r))^2 r^{1 + \sqrt{(A-2)^2 + 4 C}} e^{-2\frac{r^{(B-A)/2 +1}}{(B-A)/2 +1}} dr\notag\\
& = 2\Big(\frac{B-A+2}8\Big)^{1 + \frac{2\sqrt{(A-2)^2 +4C}}{B-A+2}} \int_0^\infty (h'(s))^2 s^{1 + \frac{4\sqrt{(A-2)^2 +4C}}{B-A+2}} e^{-\frac{s^2}2} ds.
\end{align}
Put 
\[
V(s) = \frac{s^2}2 - \big(1 + \frac{4\sqrt{(A-2)^2 +4C}}{B-A+2}\big) \ln s.
\]
It is easy to check that $V$ is convex on $(0,\infty)$ and $V''(s) \geq 1$ for all $s > 0$. Hence we obtain 
\begin{align}\label{eq:stabf2}
\int_0^\infty (h'(s))^2 s^{1 + \frac{4\sqrt{(A-2)^2 +4C}}{B-A+2}} e^{-\frac{s^2}2} ds & = \int_0^\infty (h'(s))^2 e^{-V(s)} ds\notag\\
&\geq \inf_{c\in \R} \int_0^\infty (h(s) -c)^2 e^{-V(s)} ds\notag\\
&= \inf_{c\in \R}\int_0^\infty (h(s)-c)^2 s^{1 + \frac{4\sqrt{(A-2)^2 +4C}}{B-A+2}} e^{-\frac{s^2}2} ds.
\end{align}
Using the preceding change of variable above, we deduce that
\begin{align}\label{eq:stabf3}
\int_0^\infty &(h(s)-c)^2 s^{1 + \frac{4\sqrt{(A-2)^2 +4C}}{B-A+2}} e^{-\frac{s^2}2} ds\notag\\
& = 2\Big(\frac8{B-A+2}\Big)^{2\frac{\sqrt{(A-2)^2 +4C}}{B-A+2}}\int_0^\infty(g(r) -c)^2 r^{\frac{B-A}2 + \sqrt{(A-2)^2 + 4C}} e^{-2 \frac{r^{(B-A)/2+1}}{(B-A)/2 +1}} dr\notag\\
& = 2\Big(\frac8{B-A+2}\Big)^{2\frac{\sqrt{(A-2)^2 +4C}}{B-A+2}}\int_0^\infty (f(r) - c f_0(r))^2 r^{\frac{A+B}2 -2} dr. 
\end{align}
It follows from \eqref{eq:idkey1}, \eqref{eq:idkey2},\eqref{eq:stabf1}, \eqref{eq:stabf2} and \eqref{eq:stabf3} that
\begin{align*}
&\frac{\Big(\int_0^\infty (f')^2 r^{A -1} dr + C \int_0^\infty f^2 r^{A -3} dr\Big)^{\frac12}\Big( \int_0^\infty f^2 r^{B-1} dr\Big)^{\frac12}}{\int_0^\infty f^2 r^{\frac{A+B}2-2} dr}\notag\\
&\qquad -  \frac{B-A+2 + 2\sqrt{(A-2)^2 + 4C}}4\notag\\
&\geq \frac{B-A+2}2 \inf_{c\in\R^n} \int_0^\infty (f(r) - c f_0(r))^2 r^{\frac{A+B}2 -2} dr\notag\\
&\geq \frac{B-A+2}2 \inf_{c\in\R^n, \lambda > 0} \int_0^\infty (f(r) - c f_0(\lambda r))^2 r^{\frac{A+B}2 -2} dr
\end{align*}
as desired.
\end{proof}

Finally, we shall frequently use the following one-dimensional Hardy inequality in the next sections: 
\begin{equation}\label{eq:1DHardy}
\int_0^\infty (f')^2 r^{q-1} dr \geq \frac{(q-2)^2}4 \int_0^\infty f^2 r^{q-3} dr
\end{equation}
for any function $f\in C^\infty((0,\infty))$ with 
\[
\int_0^\infty f^2 r^{q-3} dr < \infty.
\]

\section{Proof of Theorem \ref{MainII}}\label{s3}

We first consider the case $a -b + 1 > 0$. Let $u\in Y_{a,b}(\R^n)$. Since $n-2a -4 \not=0$, as discussed in Section \ref{s2}, we may assume that the spherical harmonic decomposition of $u$
\[
u(x) = u(x) = \sum_{k=0}^\infty u_k(r) \phi_k(\omega),\qquad r = |x|, \,\omega = \frac x{|x|}
\]
satisfies \eqref{eq:daohamukbehavior}, \eqref{eq:ukbehavior} and
\[
\lim_{r \to 0, \infty} r^{-2 -a +\frac n2} u_0(r) = 0.
\]
Using the identities \eqref{eq:integralDelta},  \eqref{eq:nablab}, and the Cauchy-Schwarz inequality, we have
\begin{align}\label{eq:C3abdecomp}
&\irn \frac{|\Delta u|^2}{|x|^{2a}} dx \irn \frac{|\nabla u|^2}{|x|^{2b}} dx\notag\\
& =(n\sigma_n)^2\Bigg(\sum_{k=0}^\infty \Big(\int_0^\infty (u_k'')^2r^{n-2a -1} dr+ ((n-1)(1+2a) +2c_k) \int_0^\infty (u_k')^2 r^{n-2a -3} dr\notag\\
&\qquad\qquad\qquad + c_k(c_k + 2(1+a)(n-2a -4))\int_0^\infty u_k^2 r^{n -2a -5} dr\Big)\Bigg) \times \notag\\
&\qquad \times \Bigg(\sum_{k=1}^\infty\Big(\int_0^\infty (u_k')^2 r^{n-2b -1} dr + c_k \int_0^\infty u_k^2 r^{n -2b -3} dr\Big)\Bigg)\notag\\
&\geq (n\sigma_n)^2 \Bigg(\sum_{k=0}^\infty\Big(\int_0^\infty (u_k'')^2r^{n-2a -1} dr+ ((n-1)(1+2a) +2c_k) \int_0^\infty (u_k')^2 r^{n-2a -3} dr\notag\\
&\qquad\qquad\qquad + c_k(c_k + 2(1+a)(n-2a -4))\int_0^\infty u_k^2 r^{n -2a -5} dr\Big)^{\frac12} \times \notag\\
&\qquad \times \Big(\int_0^\infty (u_k')^2 r^{n-2b -1} dr + c_k \int_0^\infty u_k^2 r^{n -2b -3} dr\Big)^{\frac12}\Bigg)^2\notag\\
&\geq (n\sigma_n)^2\Bigg(\sum_{k=0}^\infty\Big[\Bigl(\int_0^\infty (u_k'')^2r^{n-2a -1} dr+ (n-1)(1+2a) \int_0^\infty (u_k')^2 r^{n-2a -3} dr\Big)^{\frac12} \times\notag\\ 
&\qquad\qquad\qquad\qquad\qquad\qquad \times \Big(\int_0^\infty (u_k')^2 r^{n-2b -1} dr\Big)^{\frac12} + \notag\\
&\qquad \qquad \qquad + c_k\Big(2 \int_0^\infty (u_k')^2 r^{n-2a -3} dr+(c_k + 2(1+a)(n-2a -4))\int_0^\infty u_k^2 r^{n -2a -5} dr\Big)^{\frac12}\notag\\
&\qquad\qquad \qquad\qquad \Big(\int_0^\infty u_k^2 r^{n -2b -3} dr\Big)^{\frac12}\Bigl]\Bigg)^2.
\end{align}
Applying the part (i) of Lemma \ref{modelineq} for $A = n-2a, B = n-2b$ and $C = (n-1)(1+2a)$ (in this case $B-A+ 2 >0$), we obtain 
\begin{align}\label{eq:C3abpart1}
\Big(\int_0^\infty& (u_k'')^2r^{n-2a -1} dr+ (n-1)(1+2a) \int_0^\infty (u_k')^2 r^{n-2a -3} dr\Big) \int_0^\infty (u_k')^2 r^{n-2b -1} dr\notag\\
&\geq C_2(n,a,b)^2 \Big(\int_0^\infty (u_k')^2 r^{n - a -b -2} dr\Big)^2.
\end{align}
It results from  the one-dimensional Hardy inequality \eqref{eq:1DHardy} that 
\[
\int_0^\infty (u_k')^2 r^{n-2a -3} dr \geq \frac{(n-2a -4)^2}4 \int_0^\infty (u_k)^2 r^{n-2a -5} dr.
\]
Under the  condition $(n-2a-2)^2 \geq 8(1+a)(1+2a)$ and some simple computations, we have
\[
(n-2a -4)^2 + 4\Big(\frac{(n-2a-4)^2}4 + c_k + 2 (1+a)(n-2a-4)\Big) \geq (n+ 2a)^2,
\]
for any $k\geq 1$. Hence, applying the part (i) of Lemma \ref{modelineq} for $A = n-2a-2$, $B = n-2b -2$ and
\[
C = \frac{(n-2a-4)^2}4 + c_k + 2 (1+a)(n-2a-4),
\]
we get 
\begin{align}\label{eq:C3abpart2}
&\Big(2\int_0^\infty (u_k')^2 r^{n-2a -3} dr + (c_k + 2(1+a)(n-2a -4))\int_0^\infty u_k^2 r^{n -2a -5} dr\Big) \times \notag\\
&\qquad\qquad\qquad \times \int_0^\infty (u_k)^2 r^{n-2a -3} dr\notag\\
&\geq \Big(\int_0^\infty (u_k')^2 r^{n-2a -3} dr + \Big( \frac{(n-2a -4)^2}4 + c_k + 2(1+a)(n-2a -4)\Big)\int_0^\infty u_k^2 r^{n -2a -5} dr\Big) \times \notag\\
&\qquad\qquad\qquad \times \int_0^\infty (u_k)^2 r^{n-2a -3} dr\notag\\
&\geq \Big(\frac{a -b + 1 + \sqrt{2(n-2a -4)^2 + 4c_k + 8(1+a)(n-2a -4)}}2\Big)^2 \Big(\int_0^\infty u_k^2 r^{n -a -b-4} dr\Big)^2\notag\\
&\geq C_2(n,a,b)^2\Big(\int_0^\infty u_k^2 r^{n -a -b-4} dr\Big)^2.
\end{align}
Inserting \eqref{eq:C3abpart1} and \eqref{eq:C3abpart2} into \eqref{eq:C3abdecomp}, and using \eqref{eq:nablaab}, we obtain \eqref{eq:2ndCKNnew}.

It remains to verify the sharpness of \eqref{eq:2ndCKNnew}. Indeed, let choose $u = \varphi$ given in \eqref{eq:extremavphi}, which is a radial function. It is easy to see that
\[
\irn \frac{|\Delta u|^2}{|x|^{2a}} dx = n\sigma_n \Big(\int_0^\infty (u'')^2 r^{n -2a -1} dr + (n-1)(1+2a)\int_0^\infty (u')^2 r^{n -2a -3} dr\Big),
\]
\[
\irn \frac{|\na u|^2}{|x|^{2b}} dx = n \sigma_n \int_0^\infty (u')^2 r^{n -2b -1} dr,
\]
\[
\irn \frac{|\na u|^2}{|x|^{a+b+1}} dx = n \sigma_n \int_0^\infty (u')^2 r^{n -a-b -2} dr,
\]
and
\[
u'(r) = -r^{-\frac{n-2a -2 - |n+2a|}2} e^{ -\frac{r^{a-b+1}}{a-b+1}} = -r^{-\frac{A-2 -\sqrt{(A-2)^2 + 4C}}2} e^{-\frac{2 r^{\frac{B-A+2}2}}{B-A+2}}.
\]
Hence, the sharpness follows from that  of \eqref{eq:model1}.

The case $a-b + 1 < 0$ was treated by the same way with the help of the part (ii) of Lemma \ref{modelineq}. \qed

\section{Proof of Theorem \ref{MainI}}\label{s4}
In this section, we give the proof of Theorem \ref{MainI} by following the approach in \cite{DNCVPDE}. Consider the case $a -b +1 >0$. 

\noindent{\bf Step 1}.  In the first step, we shall prove an improvement of {Theorem \ref{MainI}} for curl-free vector fields $U =\nabla u\in X_{a,b}(\R^n)$ with  $u\in C^\infty(\R^n\setminus\{0\})$ that is orthogonal to all radial functions. 
\begin{lemma}\label{Improved2ndvf}
Let $n\geq 2$ and $a, b$ be real numbers such that $n-2a -4 \not=0$, $(n-2a)^2 > 4(n+1)$ and $a-b +1 >0$. Then there is a constant $\kappa >1$ depending only on $n, a$ and $b$ such that for any curl-free vector field $U =\nabla u\in X_{a,b}(\R^n)$ with $u\in C^\infty(\R^n\setminus\{0\})$ being orthogonal to all radial functions, it holds
\begin{equation*}
\irn \frac{|\nabla U|^2}{|x|^{2a}} dx \irn \frac{|U|^2}{|x|^{2b}} dx \geq \Big(\kappa C_1(n,a,b)\Big)^2 \Big(\irn \frac{|U|^2}{|x|^{a+b +1}} dx\Big)^2,
\end{equation*}
where $C_1(n,a,b)$ is given in \eqref{eq:CFLHUP}.
\end{lemma}

\begin{proof}
Since $u\in C_0^\infty(\R^n\setminus\{0\})$ is orthogonal to all radial functions,  it can be expressed in the spherical harmonic functions as
\[
u(x) = \sum_{k=1}^\infty u_k(r) \phi_k(\omega),\qquad r = |x|, \,\omega = \frac x{|x|}.
\]
Obviously, $u_k$ satisfies the conditions \eqref{eq:daohamukbehavior} and \eqref{eq:ukbehavior}. By the one-dimensional Hardy inequality \eqref{eq:1DHardy} and $c_k \geq n-1$ for any $k\geq 1$, we have
\begin{align*}
c_k \int_0^\infty &(u_k')^2 r^{n-2a -3} dr + c_k(c_k + 2(n-3a -4))\int_0^\infty u_k^2 r^{n -2a -5} dr\\
&\geq c_k\Big(\frac{(n-2a -4)^2}4 + c_k +2(n-3a -4)\Big)\int_0^\infty u_k^2 r^{n -2a -5} dr\\
&\geq c_k \Big(\frac{(n-2a -4)^2}4 + n-1 +2(n-3a -4)\Big)\int_0^\infty u_k^2 r^{n -2a -5} dr.
\end{align*}
Since $(n-2a)^2 > 4(n+1)$, it is easy to check that
\begin{equation*}
2(n -2a -4)^2 + 4(n-1) + 8(n-3a -4) > (n -2a -2)^2 + 4(n-1)
\end{equation*}
for any $k \geq 1$. So we can choose an $\epsilon > 0$ depending on $n$ and $a$ such that
\begin{equation}\label{eq:compareDelta}
(2-\epsilon)(n -2a -4)^2 + 4(n-1) + 8(n-3a -4) > (n -2a -2)^2 + 4(n-1).
\end{equation}
Applying the part (i) of Lemma \ref{modelineq} for $f= u_k'$, $A =n -2a$, $B = n-2b$ and $C = n-1 + \epsilon c_k$, we get
\begin{align*}
&\Bigg(\int_0^\infty (u_k'')^2 r^{n-2a -1} dr + (n-1 +\epsilon c_k) \int_0^\infty (u_k')^2 r^{n-2a -3} dr\Big)\int_0^\infty (u_k')^2 r^{n-2b -1} dr\notag\\
&\geq \Big(\frac{a-b+1 + \sqrt{(n-2a -2)^2 + 4(n-1 +\epsilon c_k)}}2\Big)^2 \Big(\int_0^\infty (u_k')^2 r^{n-a -b-2} dr\Big)^2\notag\\
&\geq \Big(\frac{a-b+1 + \sqrt{(n-2a -2)^2 + 4(1+\ep)(n-1)}}2\Big)^2 \Big(\int_0^\infty (u_k')^2 r^{n-a -b-2} dr\Big)^2.
\end{align*}

It results from the one-dimensional Hardy inequality \eqref{eq:1DHardy}, $c_k \geq n-1$, that 
\begin{align*}
(2-\epsilon)&\int_0^\infty (u_k')^2 r^{n-2a -3} dr + (c_k + 2(n-3a -4))\int_0^\infty u_k^2 r^{n-2a -5} dr \notag\\
&\geq \int_0^\infty (u_k')^2 r^{n-2a -3} dr + \bar{C}(n,a,\ep)\int_0^\infty u_k^2 r^{n-2a -5} dr,
\end{align*}
where \[
\bar{C}(n,a,\epsilon) :=(1-\epsilon)\frac{(n -2a -4)^2}4 + (n-1) + 2(n-3a -4).
\]
Next, applying Lemma \ref{modelineq} (or \eqref{eq:model1}) for $f = u_k$, $A = n-2a -2$, $B = n-2b-2$ and $C = \bar{C}(n,a,\ep)$ we obtain 
\begin{align*}
\Big((2-\epsilon)&\int_0^\infty (u_k')^2 r^{n-2a -3} dr + (c_k + 2(n-3a -4))\int_0^\infty u_k^2 r^{n-2a -5} dr\Big) \int_0^\infty u_k^2 r^{n-2b - 3} dr\notag\\
&\geq \Big(\int_0^\infty (u_k')^2 r^{n-2a -3} dr + \bar{C}(n,a,\ep)\int_0^\infty u_k^2 r^{n-2a -5} dr\Big) \int_0^\infty u_k^2 r^{n-2b - 3} dr\notag\\
&\geq \Big(\frac{a-b+1 + \sqrt{(n-2a -4)^2 + 4\bar{C}(n,a,\ep)}}2\Big)^2 \Big(\int_0^\infty u_k^2 r^{n-a -b - 4} dr\Big)^2.
\end{align*}
By \eqref{eq:compareDelta},  there is a constant $\kappa > 1$ depending only on $n,a$ and $b$ such that
\[
\frac{a-b+1 + \sqrt{(n-2a -4)^2 + 4\bar{C}(n,a,\ep)}}2 \geq \kappa C_1(n,a,b)
\]
and
\[
\frac{a-b+1 + \sqrt{(n-2a -2)^2 + 4(1+\ep)(n-1)}}2 \geq \kappa C_1(n,a,b).
\]
Consequently, by using the Cauchy-Schwarz inequality, we obtain
\begin{align*}
&\kappa C_1(n,a,b)\sum_{k=1}^\infty\Bigg(\int_0^\infty (u_k')^2 r^{n-a -b-2} dr + c_k \int_0^\infty u_k^2 r^{n -a-b -4} dr\Bigg)\notag\\
&\leq \sum_{k=1}^\infty\Bigg(\sqrt{\Big(\int_0^\infty (u_k'')^2 r^{n-2a -1} dr + (n-1 +\epsilon c_k) \int_0^\infty (u_k')^2 r^{n-2a -3} dr\Big)\int_0^\infty (u_k')^2 r^{n-2b -1} dr}\notag\\
&\quad + c_k\sqrt{\Big((2-\epsilon)\int_0^\infty (u_k')^2 r^{n-2a -3} dr + (c_k + 2(n-3a -4))\int_0^\infty u_k^2 r^{n-2a -5} dr\Big) \int_0^\infty u_k^2 r^{n-2b - 3} dr}\Bigg)\notag\\
&\leq \Bigg(\sum_{k=1}^\infty \Big(\int_0^\infty (u_k'')^2 r^{n-2a -1} dr + (n-1 +2c_k) \int_0^\infty (u_k')^2 r^{n-2a -3} dr \notag\\
&\qquad\qquad\qquad \qquad\qquad\qquad+ c_k(c_k + 2(n-3a -4))\int_0^\infty u_k^2 r^{n -2a -5} dr\Big)\Bigg)^{\frac12} \times \notag\\
&\qquad\qquad \times \Bigg(\sum_{k=1}^\infty\Big(\int_0^\infty (u_k')^2 r^{n-2b -1} dr + c_k \int_0^\infty u_k^2 r^{n -2b -3} dr\Big)\Bigg)^{\frac12}.
\end{align*}
Combining this with the  identities \eqref{eq:integralnabla2}, \eqref{eq:nablab} and \eqref{eq:nablaab}, we deduce that 
\begin{equation*}
\irn \frac{|\nabla U|^2}{|x|^{2a}} dx \irn \frac{|U|^2}{|x|^{2b}} dx \geq \kappa^2 C_1(n,a,b)^2 \Big(\irn \frac{|U|^2}{|x|^{a+b+1}} dx\Big)^2.
\end{equation*}
The proof of Lemma \ref{Improved2ndvf} is completed.
\end{proof}

\noindent{\bf Step 2}. For a function $u \in C^\infty(\R^n\setminus\{0\})$ with the spherical harmonic decomposition
\[
u(x) = \sum_{k=0}^\infty u_k(r) \phi_k(\omega), \quad r = |x|, \, \omega = \frac x{|x|},
\]
we write
\begin{equation}\label{edecomp}
	u = u_r + u_o,
\end{equation}
where $u_r =u_0 $ is the radial part of $u$ and $u_o = \sum_{k\geq 1} u_k \phi_k$ is orthogonal to all radial functions. We next estimate $\delta(\nabla u)$ in terms of $\delta(\nabla u_r)$ and $\delta(\nabla u_o)$. Some simple computations yield
\[
\irn \frac{|\nabla^2 u|^2}{|x|^{2a}} dx = \irn \frac{|\nabla^2 u_r|^2}{|x|^{2a}} dx + \irn \frac{|\nabla^2 u_o|^2}{|x|^{2a}} dx,
\]
\[
\irn \frac{|\nabla u|^2}{|x|^{2b}} dx = \irn \frac{|\nabla u_r|^2}{|x|^{2b}} dx + \irn \frac{|\nabla u_o|^2}{|x|^{2b}} dx
\]
and
\begin{equation}\label{eth11}
	\irn \frac{|\nabla u|^2}{|x|^{a+b+1}} dx = \irn \frac{|\nabla u_r|^2}{|x|^{a+b+1}} dx + \irn \frac{|\nabla u_o|^2}{|x|^{a+b+1}} dx.
\end{equation}
Hence, using  Cauchy-Schwarz inequality, we get
\begin{align}\label{eq:CSineq}
\Big(\irn &\frac{|\nabla^2 u|^2}{|x|^{2a}} dx \irn \frac{|\nabla u|^2}{|x|^{2b}} dx\Big)^{\frac12} \notag\\
&\geq \Big(\irn \frac{|\nabla^2 u_r|^2}{|x|^{2a}} dx \irn \frac{|\nabla u_r|^2}{|x|^{2b}} dx\Big)^{\frac12} + \Big(\irn \frac{|\nabla^2 u_o|^2}{|x|^{2a}} dx \irn \frac{|\nabla u_o|^2}{|x|^{2b}} dx\Big)^{\frac12}.
\end{align}
We now prove  the following result.
\begin{lemma}\label{decompose}
Let $n\geq 2$ and $a, b$ be real numbers satisfying $n-2a -4 \not=0$, $(n-2a)^2 > 4(n+1)$ and $a-b +1 >0$. If $U= \nabla u \in X_{a,b}(\R^n)$ with $u \in C^\infty(\R^n\setminus\{0\})$ and  $\irn \frac{|\nabla u|^2}{|x|^{a+b+1}} dx =1$, then 
\begin{equation}\label{eq:decomp1}
\delta(\nabla u) \geq \delta(\nabla u_r) \irn \frac{|\nabla u_r|^2}{|x|^{a+b+1}} dx + (\kappa -1) \irn \frac{|\nabla u_o|^2}{|x|^{a+b+1}} dx,
\end{equation} 
where $\kappa> 1$ is the constant from Lemma \ref{Improved2ndvf}. Consequently,
\begin{equation}\label{eq:orthpart}
 \irn \frac{|\nabla u_o|^2}{|x|^{a+b+1}} dx \leq \frac{1}{\kappa-1} \delta(\nabla u)
\end{equation}
and if $\delta(\nabla u) < \frac{\kappa -1}2$ then
\begin{equation}\label{eq:radialpart}
\delta(\nabla u_r) \leq 2 \delta(\nabla u).
\end{equation}
\end{lemma}
\begin{proof}
As discussed in Section \ref{s2}, we can choose the function $u$ such that its spherical harmonic decomposition satisfies \eqref{eq:daohamukbehavior}, \eqref{eq:ukbehavior} and
\[
\lim_{r\to 0, \infty} r^{\frac n2 -a -2}u(r) =0.
\]
It follows from the definition of $\delta(\nabla u)$, \eqref{eth11},  \eqref{eq:CSineq}  and $\irn \frac{|\nabla u|^2}{|x|^{a+b+1}} dx =1$ that 
\begin{align*}
	\delta(U) =\delta(\nabla u)&=\frac{\Big(\irn \frac{|\nabla^2 u|^2}{|x|^{2a}} dx \irn \frac{|\nabla u|^2}{|x|^{2b}} dx\Big)^{\frac12}}{C_1(n,a,b)} -1\\
	&=\frac{\Big(\irn \frac{|\nabla^2 u|^2}{|x|^{2a}} dx \irn \frac{|\nabla u|^2}{|x|^{2b}} dx\Big)^{\frac12}}{C_1(n,a,b)} -\irn \frac{|\nabla u_r|^2}{|x|^{a+b+1}} dx - \irn \frac{|\nabla u_o|^2}{|x|^{a+b+1}} dx\\
	&\geq \frac{\Big(\irn \frac{|\nabla^2 u_r|^2}{|x|^{2a}} dx \irn \frac{|\nabla u_r|^2}{|x|^{2b}} dx\Big)^{\frac12}}{C_1(n,a,b)} -\irn \frac{|\nabla u_r|^2}{|x|^{a+b+1}} dx \\
	&+ \frac{\Big(\irn \frac{|\nabla^2 u_o|^2}{|x|^{2a}} dx \irn \frac{|\nabla u_o|^2}{|x|^{2b}} dx\Big)^{\frac12}}{C_1(n,a,b)} - \irn \frac{|\nabla u_o|^2}{|x|^{a+b+1}} dx\\
		&\geq { \delta(\nabla u_r)}\irn \frac{|\nabla u_r|^2}{|x|^{a+b+1}} dx 
	+ \kappa \irn \frac{|\nabla u_o|^2}{|x|^{a+b+1}} dx- \irn \frac{|\nabla u_o|^2}{|x|^{a+b+1}} dx,
\end{align*}
where in the last inequality we have used the fact that $u_o$ is orthogonal to all radial functions and Lemma \ref{Improved2ndvf}.  Thus, \eqref{eq:decomp1} is proved. Notice that \eqref{eq:orthpart} is a direct consequence of \eqref{eq:decomp1}.

It remains to prove \eqref{eq:radialpart}.  Remark from \eqref{eth11} and $\irn \frac{|\nabla u|^2}{|x|^{a+b+1}} dx =1$ that 
\begin{align*}
	\irn \frac{|\nabla u_r|^2}{|x|^{a+b+1}} dx& =1- \irn \frac{|\nabla u_o|^2}{|x|^{a+b+1}} dx\\
	&\geq 1-\frac{1}{\kappa-1}{ \delta(\nabla u)}\\
	&\geq 1-\frac{1}{\kappa-1}\frac{\kappa-1}{2}=\frac 12.
\end{align*}
From this estimate and \eqref{eq:decomp1}, we deduce \eqref{eq:radialpart}.
\end{proof}

We now prove stability estimate for curl-free vector field $U =\nabla u\in X_{a,b}(\R^n)$, where $u$ is  a radial function.
\begin{theorem}\label{radialtheo}
Let $n\geq 2$ and $a, b$ be real numbers such that $n-2a -4 \not=0$, $(n-2a)^2 > 4(n+1)$ and $a-b +1 >0$. Let $U = \nabla u \in X_{a,b}(\R^n)$ such that $u\in C^\infty(\R^n\setminus\{0\})$ is a radial function and $\irn \frac{|\nabla u|^2}{|x|^{a+b+1}} dx =1$. Then,
\begin{equation*}
\delta(U) \geq (1-b+a) \inf\left\{\irn \frac{|U - V|^2}{|x|^{a+b+1}} dx\, :\, V \in \mathcal M_{a,b}\right\}
\end{equation*}
\end{theorem}
\begin{proof}
As discussed in Section \ref{s2}, we may  assume that 
\[
\lim_{r\to 0,\infty} r^{\frac n2 -a -2} u(r) =0.
\]
Since $u$ is a radial function, we have 
\[
\irn \frac{|\nabla^2 u|^2}{|x|^{2a}} dx =n\sigma_n \Bigg(\int_0^\infty (u'')^2 r^{n-2a -1} dr + (n-1) \int_0^\infty (u')^2 r^{n-2a -3} dr \Bigg),
\]
\[
\irn \frac{|\nabla u|^2}{|x|^{2b}} dx = n\sigma_n  \int_0^\infty (u')^2 r^{n-2b -1} dr
\]
and
\[
\irn \frac{|\nabla u|^2}{|x|^{a+b+1}} dx = n\sigma_n \int_0^\infty (u')^2 r^{n-a -b -2} dr.
\]
The desired estimate is derived from the part (i) of Lemma \ref{stabmodel} with parameters $A = n-2a$, $B = n-2b$ and $C = n-1$. 

\end{proof}
With  Theorem \ref{radialtheo} and Lemma \ref{decompose} at hand, we are in position to prove Theorem \ref{MainI} (in the case $1-b +a >0$).

\noindent{\bf Step 3. End of the proof of Theorem \ref{MainI}}


Let $U =\nabla u\in X_{a,b}(\R^n)$ with $u\in C^\infty(\R^n\setminus\{0\})$ and its spherical harmonic decomposition satisfies \eqref{eq:daohamukbehavior}, \eqref{eq:ukbehavior} and
\[
\lim_{r\to 0,\infty} r^{\frac n2 -a -2} u_0(r) =0.
\]
Without loss of generality, we can assume that $\irn \frac{|\nabla u|^2}{|x|^{a+b+1}} dx =1$. We use again the decomposition \eqref{edecomp}.

We first consider the case $\delta(U) = \delta(\nabla u) <\frac{\kappa-1}{2}$where $\kappa>1$ given in  Lemma \ref{decompose}. It follows from Lemma \ref{decompose} that $
\delta(\nabla u_r) \leq 2 \delta(\nabla u)$ and 
$$ \irn \frac{|\nabla u_o|^2}{|x|^{a+b+1}} dx \leq \frac{1}{\kappa -1} \delta(\nabla u).$$
 Applying Theorem \ref{radialtheo} for $u_r$, we get
$$ \inf\left\{\irn \frac{|\nabla u_r - V|^2}{|x|^{a+b+1}} dx\, :\, V \in \mathcal M_{a,b}\right\}\leq \frac{1}{1-b+a}\delta(\nabla u_r) \leq \frac{2}{1-b+a} \delta(\nabla u).$$
 By Cauchy-Schwarz inequality, we have for $V \in \mathcal M_{a,b}$ that
\[
\irn \frac{|\nabla u - V|^2}{|x|^{a+b+1}} dx \leq 2\lt(\irn \frac{|\nabla u_o|^2}{|x|^{a+b+1}} dx +\irn \frac{|\nabla u_r - V|^2}{|x|^{a+b+1}} dx\rt).
\]
Combining these estimates, we arrive at
\begin{equation}\label{eb1}
	 \inf\left\{\irn \frac{|\nabla u - V|^2}{|x|^{a+b+1}} dx\, :\, V \in \mathcal M_{a,b}\right\}  \leq 2\Big(\frac1{\kappa -1} + \frac2{1-b+a}\Big) \de(\nabla u)=2C \delta(\nabla u),
\end{equation}
with 
\[
C = \frac1{\kappa -1} + \frac2{1-b+a}.
\]
We now show that the infimum above is attained by a curl-free vector field  $V\in \mathcal M_{a,b}$ if $\delta(\nabla u) < \frac1{16 C}<\frac{\kappa-1}{2}.$ Indeed, let $\varphi_i(x) = c_i \varphi(\sqrt{\lambda_i} x)$, $c_i\in \R, \lambda_i>0$, $i\geq 2$ ($\varphi$ is given in \eqref{eq:extremavphi}) be a sequence of functions such that $V_i =\nabla \varphi_i\in \mathcal M_{a,b}$ and
$$\lim_{i\to \infty} \irn \frac{|\na u -\na \varphi_i|^2}{|x|^{a+b+1}} dx = \inf_{V \in \mathcal{M}_{a,b}} \irn \frac{|\na u -V|^2}{|x|^{a+b+1}} dx < \frac{2C}{16C}=\frac18.$$
The  triangle inequality implies that 
$$\Big(\irn \frac{|\na u -\na \varphi_i|^2}{|x|^{a+b+1}} dx\Big)^{\frac12} \geq \Big| \Big(\irn \frac{|\na \varphi_i|^2}{|x|^{a+b+1}} dx\Big)^{\frac12} -1\Big|.$$
Thus for $i$ large enough, we get
$$\frac14 \leq \irn \frac{|\na \varphi_i|^2 }{|x|^{a+b+1}}dx \leq \frac 94.$$
A simple computation gives
$$\irn \frac{|\na \varphi_i|^2 }{|x|^{a+b+1}}dx = c_i^2 \lam_i^{-\frac n2 +1+\frac{a+b+1}{2}} \irn\frac{( \varphi'(|x|))^2 }{|x|^{a+b+1}}dx.$$
Thus, there are $D_1, D_2 >0$ such that $D_1^2 \leq c_i^2 \lam_i^{-\frac n2 + 1+\frac{a+b+1}{2}}\leq D_2^2$ for $i$ large enough. We also note that
\begin{multline}\label{eq:expan}
	\irn\frac{ |\na u -\na \varphi_i|^2}{|x|^{a+b+1}} dx \\= 1 + c_i^2 \lam_i^{-\frac n2 + 1+\frac{a+b+1}{2}}\irn \frac{( \varphi'(|x|))^2 }{|x|^{a+b+1}}dx -2c_i \lam_i^{\frac12} \irn \frac{\na u(x) \cdot \frac x{|x|} \varphi'(\sqrt{\lambda_i}|x|)}{|x|^{a+b+1}}dx.
\end{multline}
We next claim that $(\lam_i)$ is bounded from above and below by positive constants. Indeed, suppose, up to extracting a subsequence, that
$$\lim_{i\to \infty} \lam_i = \infty.$$
For any $\epsilon >0$, there exists $R>0$ small enough such that $\int_{B_R}\frac{ |\na u|^2}{|x|^{a+b+1}} dx < \ep^2$. Hence, 
\begin{align*}
	&\Big| \irn \frac{\na u(x) \cdot \frac{x}{|x|} \varphi'(\sqrt{\lambda_i}|x|)}{|x|^{a+b+1}}dx\Big|\leq \int_{B_R} \frac{|\na u||\varphi'(\sqrt{\lambda_i}|x|)|}{|x|^{a+b+1}} dx + \int_{B_R^c}\frac{ |\na u| \varphi'(\sqrt{\lambda_i}|x|)}{|x|^{a+b+1}}dx\\
	&\leq \Big(\int_{B_R^c} \frac{|\na u|^2}{|x|^{a+b+1}} dx\Big)^{\frac12} \Big(\int_{B_R^c}\frac{|\varphi'(\sqrt{\lambda_i}|x|)|^2}{|x|^{a+b+1}}dx\Big)^{\frac12} + \ep \Big(\int_{B_R}\frac{(|\varphi'(\sqrt{\lambda_i}|x|)|)^2}{|x|^{a+b+1}} dx\Big)^{\frac12}\\
	&\leq \lam_i^{\frac {-n+a+b+1}{4}}\Big(\int_{B_{R\sqrt{\lam_i}}^c}\frac{(\varphi'(|x|))^2}{|x|^{a+b+1}}dx\Big)^{\frac12} + \lam_i^{\frac {-n+a+b+1}{4}} \ep  \Big(\int_{B_{R\sqrt{\lam_i}}}\frac{(\varphi'(|x|))^2}{|x|^{a+b+1}} dx\Big)^{\frac12},
\end{align*}
which implies
$$
\Big|c_i \lam_i^{\frac 12}\irn \frac{\na u(x) \cdot \frac{x}{|x|} \varphi'(\sqrt{\lambda_i}|x|)}{|x|^{a+b+1}}dx\Big| \leq D_2 \Big(\int_{B_{R\sqrt{\lam_i}}^c}\frac{(\varphi'(|x|))^2}{|x|^{a+b+1}} dx\Big)^{\frac12} + D_2\ep  \Big(\int_{B_{R\sqrt{\lam_i}}}\frac{(\varphi'(|x|))^2}{|x|^{a+b+1}} dx\Big)^{\frac12}.
$$
Let $i\to \infty$, we obtain
$$\limsup_{i\to \infty}\Big|c_i \lam_i^{\frac 12}\irn \frac{\na u \cdot (\frac{x}{|x|}) \varphi'(\sqrt{\lambda_i}|x|)}{|x|^{a+b+1}}dx\Big| \leq D_2 \Big(\irn \frac{(\varphi'(|x|))^2}{|x|^{a+b+1}}  dx\Big)^{\frac12} \ep.$$
Since $\ep>0$ is arbitrary, we get
$$\limsup_{i\to \infty}\Big|c_i \lam_i^{\frac 12}\irn \frac{\na u \cdot (\frac{x}{|x|}) \varphi'(\sqrt{\lambda_i}|x|)}{|x|^{a+b+1}} dx \Big|=0.$$
This together with \eqref{eq:expan} implies
$$\frac18 \geq \lim_{i\to \infty} \irn \frac{|\na u -\na \varphi_i|^2}{|x|^{a+b+1}} dx > 1,$$
which is impossible.{ Concerning the lower bound of $\lambda_i$,}  suppose, up to extracting a subsequence, that
$$\lim_{i\to \infty} \lam_i = 0.$$
For any $\epsilon >0$, there exists $R>0$ large enough such that $\int_{B_R^c}\frac{ |\na u|^2 }{|x|^{a+b+1}}dx < \ep^2$. By using the similar arguments as in the previous case, we also obtain
$$\limsup_{i\to \infty}\Big|c_i \lam_i^{\frac 12}\irn \frac{\na u \cdot (\frac{x}{|x|}) \varphi'(\sqrt{\lambda_i}|x|)}{|x|^{a+b+1}} dx \Big|=0.$$
This together with \eqref{eq:expan} implies
$$\frac18 \geq \lim_{i\to \infty} \irn \frac{|\na u -\na \varphi_i|^2}{|x|^{a+b+1}} dx > 1,$$
which is again impossible. So the claim is proved. As a consequence, there are $a_1, a_2, a_3 >0$ such that 
$$|c_i|\leq a_1,\qquad 0< a_2 \leq \lam_i \leq a_3$$
for any $i$. Extracting a subsequence, we can assume that $c_i \to c$ and $\lam_i \to \lam \in [a_2,a_3]$. Denote $\varphi_0 = c \varphi(\lambda x)$ and $V_0=\nabla \varphi_0\in \mathcal M_{a,b}$. We then have $\na \varphi_i \to \na \varphi_0$ in $L^2(|x|^{-\frac{a+b+1}{2}}dx)$ and hence
\begin{align*}
\irn \frac{|\na u- V_0|^2}{|x|^{a+b+1}} dx =\irn \frac{|\na u- \na \varphi_0|^2}{|x|^{a+b+1}} dx& = \lim_{i\to \infty} \irn\frac{ |\na u -\na \varphi_i|^2 }{|x|^{a+b+1}}dx\\
& = \inf_{V \in \mathcal M_{a,b}} \irn\frac{ |\na u -V|^2}{|x|^{a+b+1}} dx.
\end{align*}
Thus, the infimum in \eqref{eb1} is attained.
Remark that 
$$\irn\frac{ |\na u|^2}{|x|^{a+b+1}} dx =1,\quad \irn\frac{ |\na u- V_0|^2 }{|x|^{a+b+1}}dx = \inf_{V \in \mathcal M_{a,b}} \irn \frac{|\na u -V|^2 }{|x|^{a+b+1}}dx \leq \frac 18$$ 
then $V_0\not\equiv 0$. So we can choose a positive constant $\bar{a}$ such that 
$$\irn\frac{ |\bar{a} V_0|^2}{|x|^{a+b+1}} dx = \irn\frac{ |\na u|^2}{|x|^{a+b+1}} dx = 1$$
or, equivalently $\bar{a} = (\irn \frac{|V_0|^2}{|x|^{a+b+1}} dx)^{-\frac12}$. The triangle inequality follows that 
\begin{align*}
	\Big|\Big(\irn\frac{ |V_0|^2}{|x|^{a+b+1}} dx\Big)^{\frac12} -1\Big| &= \Big|\Big(\irn\frac{ |V_0|^2}{|x|^{a+b+1}} dx\Big)^{\frac12} -\Big(\irn \frac{|\na u|^2}{|x|^{a+b+1}} dx\Big)^{\frac12}\Big|\\
	&\leq \Big(\irn \frac{|\na u-V_0|^2}{|x|^{a+b+1}} dx\Big)^{\frac12} \\
	&\leq \big(2C\delta(\nabla u)\big)^{\frac 12},
\end{align*}
where in the last inequality, we have used \eqref{eb1}.
From this estimate and the Cauchy-Schwartz inequality, we obtain
\begin{align*}
	\irn\frac{ |\na u - \bar{a} V_0|^2 }{|x|^{a+b+1}}dx &= \irn\frac{ |\na u -V_0  + (1-\bar{a})V_0|^2}{|x|^{a+b+1}} dx\notag\\
	&\leq 2 \irn\frac{ |\na u -V_0|^2 }{|x|^{a+b+1}}dx + 2 (\bar{a}-1)^2 \irn \frac{|V_0|^2}{|x|^{a+b+1}} dx\notag\\
	&\leq C\delta(\nabla u) + 2 \Big|\Big(\irn \frac{|\na \varphi|^2 }{|x|^{a+b+1}}dx\Big)^{\frac12} -1\Big|^2\notag\\
	&\leq 8C\delta(\nabla u).
\end{align*}
This gives
\begin{equation*}
\delta(\nabla u) \geq \frac1{8C} \inf\left\{\frac{\irn \frac{|\nabla u - V|^2}{|x|^{a+b+1}} dx}{\irn \frac{|\nabla u |^2}{|x|^{a+b+1}} dx}\, :\, V \in \mathcal M_{a,b},\quad \irn \frac{|\nabla u |^2}{|x|^{a+b+1}} dx = \irn \frac{|V|^2}{|x|^{a+b+1}} dx\right\}
\end{equation*}
when $\delta (\nabla u) < \frac{1}{16C}$. Hence,  Theorem \ref{MainI} is proved for $\delta (\nabla u) < \frac{1}{16C}$.

It remains to prove the theorem when  $\delta (\nabla u) \geq \frac{1}{16 C}$. Observing that 
$$\inf\left\{\frac{\irn \frac{|\nabla u - V|^2}{|x|^{a+b+1}} dx}{\irn \frac{|\nabla u |^2}{|x|^{a+b+1}} dx}\, :\, V \in \mathcal M_{a,b},\quad \irn \frac{|\nabla u |^2}{|x|^{a+b+1}} dx = \irn \frac{|V|^2}{|x|^{a+b+1}} dx\right\}\leq 4.$$
Then, 
\begin{equation*}
	\delta(u) \geq \frac{1}{64C} \inf\left\{\frac{\irn \frac{|\nabla u - V|^2}{|x|^{a+b+1}} dx}{\irn \frac{|\nabla u |^2}{|x|^{a+b+1}} dx}\, :\, V \in \mathcal M_{a,b},\quad \irn \frac{|\nabla u |^2}{|x|^{a+b+1}} dx = \irn \frac{|V|^2}{|x|^{a+b+1}} dx\right\}
\end{equation*}
as desired. 
Thus, we finish the proof of Theorem \ref{MainI} in the case $a -b + 1> 0$. By using the same arguments, we also arrive at the conclusion for the case $a-b+1<0$. The details are omitted. \qed

\section{Proof of Theorem \ref{MainIII}}\label{s5}
The proof of Theorem \ref{MainIII} is similar to that of Theorem \ref{MainI}. We need to  deal with the case $a -b + 1 > 0$ and the case $a -b + 1< 0$ is treated by the same arguments. We first prove an improvement of \eqref{eq:2ndCKNnew} for functions $u\in Y_{a,b}(\R^n))$ that are orthogonal to all radial functions.
\begin{lemma}\label{Improved2ndCKN1}
Let $n\geq 2$ and $a, b$ be real numbers such that $n-2a -4\not=0$, $(n-2a -2)^2 > 8(1+a)(1+2a)$ and $a-b +1 >0$. Then there is a constant ${\bar{\kappa} >1}$ depending only on $n, a$ and $b$ such that for any function $u\in Y_{a,b}(\R^n)$ that is orthogonal to all radial functions, it holds
\begin{equation*}
\irn \frac{|\Delta u|^2}{|x|^{2a}} dx \irn \frac{|\nabla u|^2}{|x|^{2b}} dx \geq \Big(\bar{\kappa} C_2(n,a,b)\Big)^2 \Big(\irn \frac{|\nabla u|^2}{|x|^{a+b +1}} dx\Big)^2
\end{equation*}
where $C_2(n,a,b)$ is given in \eqref{eq:2ndCKNnew}.
\end{lemma}
\begin{proof}
Let $u\in Y_{a,b}(\R^n)$ be a function that is orthogonal to all radial functions. Then, 
\[
u(x) = u(x) = \sum_{k=1}^\infty u_k(r) \phi_k(\omega),\qquad r = |x|, \,\omega = \frac x{|x|}.
\]
Furthermore $u_k$ satisfies \eqref{eq:daohamukbehavior} and \eqref{eq:ukbehavior}. Note that the condition $(n-2a -2)^2 > 8(1+a)(1+2a)$ is equivalent to
\[
2(n-2a-4)^2 + 4(n-1) + 8(1+a)(n-2a -4) > (n+2a)^2.
\]
Hence, we can choose an $\epsilon >0$ depending only on $n$ and $a$ such that
\begin{equation}\label{eq:chooseep}
(2-\epsilon)(n-2a-4)^2 + 4(n-1) + 8(1+a)(n-2a -4) > (n+2a)^2.
\end{equation}
Denote
\[
 \tilde{C}(n,a,\epsilon) : = (1-\epsilon)\frac{(n-2a-4)^2}4 + (n-1) + 2(1+a)(n-2a -4).
\]
Note that $c_k \geq c_1 = n-1$ for any $k\geq 1$. By using one dimensional Hardy inequality \eqref{eq:1DHardy}, we have
\begin{align*}
\irn &\frac{|\Delta u|^2}{|x|^{2a}} dx =n\sigma_n \sum_{k=1}^\infty\Bigg(\int_0^\infty (u_k'')^2 r^{n-2a -1} dr + ((n-1)(1+2a) +2c_k) \int_0^\infty (u_k')^2 r^{n-2a -3} dr\notag\\
&\qquad\qquad\qquad + c_k(c_k + 2(1+a)(n-2a -4))\int_0^\infty u_k^2 r^{n -2a -5} dr\Bigg)\\
&=n\sigma_n \sum_{k=1}^\infty\Bigg(\Big(\int_0^\infty (u_k'')^2 r^{n-2a -1} dr + ((n-1)(1+2a) +\epsilon c_k)\int_0^\infty (u_k')^2 r^{n-2a -3} dr\Big) +\notag\\
&\qquad +c_k\Big((2-\epsilon) \int_0^\infty (u_k')^2 r^{n-2a -3} dr + (c_k + 2(1+a)(n-2a -4))\int_0^\infty u_k^2 r^{n -2a -5} dr\Big)\Bigg)\notag\\
&\geq n\sigma_n \sum_{k=1}^\infty\Bigg(\Big(\int_0^\infty (u_k'')^2 r^{n-2a -1} dr + ((n-1)(1+2a+\epsilon))\int_0^\infty (u_k')^2 r^{n-2a -3} dr\Big) +\notag\\
&\qquad +c_k\Big(\int_0^\infty (u_k')^2 r^{n-2a -3} dr + \tilde C(n,a,\epsilon)\int_0^\infty u_k^2 r^{n -2a -5} dr\Big)\Bigg).
\end{align*}
Using  the Cauchy-Schwarz inequality, we get
\begin{align}\label{eq:decompose2ndCKN1}
&\irn \frac{|\Delta u|^2}{|x|^{2a}} dx \irn \frac{|\nabla u|^2}{|x|^{2b}} dx\notag\\
&\geq (n\sigma_n)^2 \Bigg(\sum_{k=1}^\infty\Bigg(\Big(\int_0^\infty (u_k'')^2 r^{n-2a -1} dr + ((n-1)(1+2a+\epsilon))\int_0^\infty (u_k')^2 r^{n-2a -3} dr\Big) +\notag\\
&\qquad +c_k\Big(\int_0^\infty (u_k')^2 r^{n-2a -3} dr +  \tilde{C}(n,a,\epsilon)\int_0^\infty u_k^2 r^{n -2a -5} dr\Big)\Bigg)\Bigg)\times\notag\\
&\qquad\qquad \qquad \times \Bigg(\sum_{k=1}^\infty\Big(\int_0^\infty (u_k')^2 r^{n-2b -1} dr + c_k \int_0^\infty u_k^2 r^{n -2b -3} dr\Big)\Bigg)\notag\\
&\geq (n\sigma_n)^2 \Bigg(\sum_{k=1}^\infty\Big(\Big(\int_0^\infty (u_k'')^2 r^{n-2a -1} dr + (n-1)(1+2a+\epsilon)\int_0^\infty (u_k')^2 r^{n-2a -3} dr\Big) +\notag\\
&\qquad\qquad +c_k\Big(\int_0^\infty (u_k')^2 r^{n-2a -3} dr +  \tilde{C}(n,a,\epsilon)\int_0^\infty u_k^2 r^{n -2a -5} dr\Big)\Big)^{\frac12}\times \notag\\
&\qquad\qquad\qquad \times \Big(\int_0^\infty (u_k')^2 r^{n-2b -1} dr + c_k \int_0^\infty u_k^2 r^{n -2b -3} dr\Big)^{\frac12}\Bigg)^2\notag\\
&\geq (n\sigma_n)^2 \Bigg(\sum_{k=1}^\infty\Big(\Big(\int_0^\infty (u_k'')^2 r^{n-2a -1} dr + (n-1)(1+2a+\epsilon)\int_0^\infty (u_k')^2 r^{n-2a -3} dr\Big)^{\frac12} \times\notag\\
&\qquad\qquad\qquad \times \Big(\int_0^\infty (u_k')^2 r^{n-2b -1} dr\Big)^{\frac12}+ \notag\\
&\qquad\qquad + c_k\Big(\int_0^\infty (u_k')^2 r^{n-2a -3} dr +  \tilde{C}(n,a,\epsilon)\int_0^\infty u_k^2 r^{n -2a -5} dr\Big)^{\frac12} \times\notag\\
&\qquad\qquad\qquad \times\Big(\int_0^\infty u_k^2 r^{n -2b -3} dr\Big)^{\frac12}\Big) \Bigg)^2.
\end{align}
Applying the part (i) of Lemma \ref{modelineq} for $f = u_k'$, $A = n-2a$, $B =n-2b$ and $C = (n-1)(1 + 2a +\epsilon)$ we have (note that $B-A +2 = 2(a-b +1) >0$)
\begin{align*}
&\Big(\int_0^\infty (u_k'')^2 r^{n-2a -1} dr + (n-1)(1+2a+\epsilon)\int_0^\infty (u_k')^2 r^{n-2a -3} dr\Big)\int_0^\infty (u_k')^2 r^{n-2b -1} dr\notag\\
&\geq \Big(\frac{a-b+1 + \sqrt{(n+2a)^2 + 4\ep(n-1)}}2\Big)^2 \Big(\int_0^\infty (u_k')^2 r^{n-a-b -2} dr\Big)^2.
\end{align*}
Using again the part (i) of Lemma \ref{modelineq} for $f = u_k$, $A = n-2a-2$, $B =n-2b-2$ and $C = \tilde C(n,a,\epsilon)$ we have
\begin{align*}
&\Big(\int_0^\infty (u_k')^2 r^{n-2a -2} dr +  \tilde{C}(n,a,\epsilon)\int_0^\infty (u_k)^2 r^{n-2a -5} dr\Big)\int_0^\infty (u_k)^2 r^{n-2b -3} dr\notag\\
&\geq \Big(\frac{a-b+1 + \sqrt{(n-2a-4)^2 + 4 \tilde{C}(n,a,\epsilon)}}2\Big)^2 \Big(\int_0^\infty (u_k)^2 r^{n-a-b -4} dr\Big)^2.
\end{align*}
It results from the choice of $\epsilon >0$ (see \eqref{eq:chooseep}) that 
\[
\frac{a-b+1 + \sqrt{(n+2a)^2 + 4\ep(n-1)}}2 > C_2(n,a,b)
\]
and 
\[
\frac{a-b+1 + \sqrt{(n-2a-4)^2 + 4 \tilde{C}(n,a,\epsilon)}}2 > C_2(n,a,b).
\]
Hence, there is a constant $\bar{\kappa} > 1$ depending only on $n,a$ and $b$ such that
\begin{align*}
&\Big(\int_0^\infty (u_k'')^2 r^{n-2a -1} dr + (n-1)(1+2a+\epsilon)\int_0^\infty (u_k')^2 r^{n-2a -3} dr\Big)\int_0^\infty (u_k')^2 r^{n-2b -1} dr\notag\\
&\geq \Big(\bar{\kappa} C_2(n,a,b)\Big)^2 \Big(\int_0^\infty (u_k')^2 r^{n-a-b -2} dr\Big)^2,
\end{align*}
and
\begin{align*}
&\Big(\int_0^\infty (u_k')^2 r^{n-2a -2} dr +  \tilde{C}(n,a,\epsilon)\int_0^\infty (u_k)^2 r^{n-2a -5} dr\Big)\int_0^\infty (u_k)^2 r^{n-2b -3} dr\notag\\
&\geq \Big(\bar{\kappa} C_2(n,a,b)\Big)^2 \Big(\int_0^\infty (u_k)^2 r^{n-a-b -4} dr\Big)^2.
\end{align*}
Inserting these two estimates into \eqref{eq:decompose2ndCKN1}, we obtain
\begin{align*}
&\irn \frac{|\Delta u|^2}{|x|^{2a}} dx \irn \frac{|\nabla u|^2}{|x|^{2b}} dx\notag\\
&\geq (\bar{\kappa}C_2(n,a,b)\Big)^2 (n\sigma_n)^2 \Bigg(\sum_{k=1}^\infty\Big(\int_0^\infty (u_k')^2 r^{n-a-b -2} dr + c_k \int_0^\infty (u_k)^2 r^{n-a-b -4} dr\Big)\Bigg)^2\notag\\
&=\Big(\bar{\kappa}C_2(n,a,b))^2 \Big(\irn \frac{|\nabla u|^2}{|x|^{a+b+1}} dx\Big)^2,
\end{align*}
where we have used  \eqref{eq:nablaab}. The proof of this lemma is completed.
\end{proof}

For a function $u \in C^\infty(\R^n\setminus\{0\})$ we write $u = u_r + u_o$, where $u_r $ is the radial part of $u$ and $u_o = \sum_{k=1}^\infty u_k \phi_k$ is orthogonal to all radial functions. It is easy to check that
\[
\irn \frac{|\Delta u|^2}{|x|^{2a}} dx = \irn \frac{|\Delta u_r|^2}{|x|^{2a}} dx + \irn \frac{|\Delta u_o|^2}{|x|^{2a}} dx,
\]
\[
\irn \frac{|\nabla u|^2}{|x|^{2b}} dx = \irn \frac{|\nabla u_r|^2}{|x|^{2b}} dx + \irn \frac{|\nabla u_o|^2}{|x|^{2b}} dx
\]
and
\[
\irn \frac{|\nabla u|^2}{|x|^{a+b+1}} dx = \irn \frac{|\nabla u_r|^2}{|x|^{a+b+1}} dx + \irn \frac{|\nabla u_o|^2}{|x|^{a+b+1}} dx.
\]
Hence, by using the Cauchy-Schwarz inequality, we have
\begin{align*}
\Big(\irn &\frac{|\Delta u|^2}{|x|^{2a}} dx \irn \frac{|\nabla u|^2}{|x|^{2b}} dx\Big)^{\frac12} \notag\\
&\geq \Big(\irn \frac{|\Delta u_r|^2}{|x|^{2a}} dx \irn \frac{|\nabla u_r|^2}{|x|^{2b}} dx\Big)^{\frac12} + \Big(\irn \frac{|\Delta u_o|^2}{|x|^{2a}} dx \irn \frac{|\nabla u_o|^2}{|x|^{2b}} dx\Big)^{\frac12}.
\end{align*}
As a consequence, we obtain the following lemma.
\begin{lemma}\label{decomposeDelta}
Let $n\geq 2$ and $a, b$ be real numbers such that $n-2a -4\not=0$, $(n-2a -2)^2 > 8(1+a)(1+2a)$ and $a-b +1 >0$. Let $u \in Y_{a,b}(\R^n)$ satisfying  $\irn \frac{|\nabla u|^2}{|x|^{a+b+1}} dx =1$. Then,
\begin{equation*}
\delta_{CKN}(u) \geq \delta_{CKN}(u_r) \irn \frac{|\nabla u_r|^2}{|x|^{a+b+1}} dx + (\bar{\kappa} -1) \irn \frac{|\nabla u_o|^2}{|x|^{a+b+1}} dx,
\end{equation*} 
where $\bar{\kappa}>1$ is the constant from Lemma \ref{Improved2ndCKN1}. Consequently,
\begin{equation*}
 \frac{|\nabla u_o|^2}{|x|^{a+b+1}} dx \leq \frac{1}{\bar{\kappa} -1} \delta_{CKN}(u)
\end{equation*}
and if $\delta_{CKN}(u) < \frac{\bar{\kappa} -1}2$ then
\begin{equation*}
\delta_{CKN}(u_r) \leq 2 \delta_{CKN}(u).
\end{equation*}
\end{lemma}
The proof of this lemma is similar to that of Lemma \ref{decompose}, so we omit the details.

Next, we establish the stability estimate for radial functions in $Y_{a,b}(\R^n)$ as follows.
\begin{theorem}\label{radialtheoDelta}
Let $n\geq 2$ and $a, b$ be real numbers such that $n-2a -4\not=0$, $(n-2a -2)^2 > 8(1+a)(1+2a)$ and $a-b +1 >0$. Let $u \in Y_{a,b}(\R^n)$ be a radial function such that $\irn \frac{|\nabla u|^2}{|x|^{a+b+1}} dx =1$. Then,
\begin{equation*}
\delta_{CKN} (u) \geq (1-b +a) \inf\left\{\irn \frac{|\nabla u - \nabla v|^2}{|x|^{a+b+1}} dx\, :\, v \in \mathcal N_{a,b}\right\}.
\end{equation*}
\end{theorem}
\begin{proof}
As above, we  may  assume that
\[
\lim_{r\to 0, \infty} r^{\frac n2 -a -2} u(r) =0
\]
thanks to  $n-2a-4 \not=0$ (see Section \ref{s2}). Moreover, since $u$ is a radial function, we get 
\[
\irn \frac{|\Delta u|^2}{|x|^{2a}} dx =n\sigma_n \Bigg(\int_0^\infty (u'')^2 r^{n-2a -1} dr + (n-1)(1+2a) \int_0^\infty (u')^2 r^{n-2a -3} dr \Bigg),
\]
\[
\irn \frac{|\nabla u|^2}{|x|^{2b}} dx = n\sigma_n  \int_0^\infty (u')^2 r^{n-2b -1} dr
\]
and
\[
\irn \frac{|\nabla u|^2}{|x|^{a+b+1}} dx = n\sigma_n \int_0^\infty (u')^2 r^{n-a -b -2} dr.
\]
The desired estimate follows from the part (i) of Lemma \ref{stabmodel} with $f = u_k'$, $A =n -2a$, $B = n-2b$ and $C =(n-1)(1 +2a)$.
\end{proof}
With Lemma \ref{decomposeDelta} and Theorem \ref{radialtheoDelta} at hand, it is sufficient to mimic the last part of the proof of Theorem \ref{MainI} to finish the proof of Theorem \ref{MainIII} in the case $1-b+a >0$. The details are also omitted. \qed

\section*{Acknowledgments}
This work was initiated and done when the second author visited Vietnam Institute for Advanced Study in Mathematics (VIASM). He would like to thank the institute for hospitality and support during the visits. This work was supported by Vietnam National Foundation for Science and Technology Development (NAFOSTED)[101.02-2025.33].
\bibliographystyle{abbrv}

\end{document}